\documentclass[12pt]{amsart}
\usepackage{inputenc}
\usepackage{amsmath, latexsym, amsfonts, amssymb, amsthm, amscd}
\usepackage[left=3cm, right=2.5cm, top=2.5cm, bottom=2.5cm]{geometry}
\usepackage{color}
\usepackage{ulem}

\newtheorem{proposition}{Proposition}
\newtheorem{lemma}{Lemma}

\newtheorem{theorem}{Theorem}

\renewcommand{\P}{\mathbb{P}}

\newcommand{\ud}{\mathrm{d}}

\newcommand{\R}{\mathbb{R}}

\newcommand{\eps}{\varepsilon}

\newcommand{\Expo}[1]{\exp\left\{#1\right\}}
\newcommand{\E}{\mathbb{E}}

\def\bea{\begin{eqnarray}}
\def\eea{\end{eqnarray}} 
\def\Bea{\begin{eqnarray*}}
\def\Eea{\end{eqnarray*}} 

\usepackage{tikz}

\begin{document}

 \title[Extinction of CSBP in critical environment]{Extinction rate of continuous state branching processes in critical 
 L\'evy environments}

\author{Vincent Bansaye}
\address{CMAP, \'Ecole Polytechnique, Route de Saclay, F-91128 Palaiseau Cedex, France}
\email{vincent.bansaye@polytechnique.edu}

\author{Juan Carlos Pardo}
\address{CIMAT A.C. Calle Jalisco s/n. C.P. 36240, Guanajuato, Mexico}
\email{jcpardo@cimat.mx}

\author{Charline Smadi}
\address{Universit\'e Clermont Auvergne, Irstea, UR LISC, Centre de Clermont-Ferrand, 
9 avenue Blaise Pascal CS 20085, F-63178 Aubi\`ere, France and Complex Systems Institute of Paris Ile-de-France, 
113 rue Nationale, Paris, France}
\email{charline.smadi@irstea.fr}

 \maketitle
\vspace{0.2in}


\bigskip


\vspace{0.5cm}

\begin{abstract}
We study  the  speed  of extinction of continuous state branching processes in a L\'evy environment, where the associated L\'evy 
process oscillates.  Assuming that the  L\'evy process satisfies  Spitzer's condition and the existence of some 
exponential moments, we extend recent results
where the associated branching mechanism is stable. The study  relies on the  path analysis of  the branching process  together with its L\'evy
environment, when the latter is conditioned to have a non-negative running infimum. For that purpose, 
we combine the  approach  developed in  
 Afanasyev et al. \cite{Afanasyev2005},  for the discrete
setting and i.i.d. environments, with fluctuation theory of L\'evy processes and a result on exponential functionals of L\'evy processes due to Patie and Savov \cite{patie2016bernstein}.
 \end{abstract}

\tableofcontents

\section{Introduction and main results}

In this manuscript, we are interested in continuous state branching processes (CSBPs) which can be considered as  the continuous analogues of 
 Galton-Watson (GW) processes in  time and  state space. 
  Formally speaking, a process in this class is a strong Markov process taking values in $[0,\infty]$, where $0$ and 
  $\infty$ are absorbing states, and satisfying the branching property, that is to say the law of the process started from $x+y$ is the same  as the law of the  sum  of two independent copies of the same process issued respectively from $x$ and $y$. CSBPs have been introduced by Jirina \cite{MR0101554} in the late fifties, of the last 
  century, and since then they have been deeply studied by many authors  including Bingham \cite{MR0410961}, Grey \cite{MR0408016}, Grimvall \cite{MR0362529}, 
  Lamperti \cite{MR0208685,MR0217893}, to name but a few.  An interesting feature of CSBPs is that they can be obtained as 
  scaling limits of GW processes, see for instance Grimvall  \cite{MR0362529} and Lamperti \cite{MR0217893}.

Galton-Watson processes in random environment (GWREs) were introduced  by 
Smith and Wilkinson \cite{smith1969branching} in the late sixties of the last 
  century. This type of processes has attracted considerable interest in the last
decade, see for instance  \cite{afanasyev2012limit,Afanasyev2005,babe,bo10} and the references therein. 
Indeed, such a family of processes provides a richer class of population models, taking into account the effect of the environment on demographic parameters and letting new phenomena 
appear.  In particular, the  classification of the asymptotic behaviour of rare events, such as survival probability and  large deviations, 
is much more complex, since it may combine environmental and demographical stochasticities.

 Scaling limits of GWREs have been studied by 
Kurtz \cite{Kurtz} in the  continuous path setting and, more recently, by Bansaye and Simatos \cite{bansima}  and Bansaye et al. \cite{bansaye2018scaling}  
 for a larger class of processes in random environment that includes CSBPs.  The limiting processes satisfy the Markov property and the \emph{quenched} branching property, i.e. conditionally on the environment the process started from $x+y$ is distributed as the sum  of two independent copies of the same process issued respectively from $x$ and $y$. 
Such processes may be thought of, and therefore called,    {\it CSBPs in  random environment}. An interesting  subclass of the aforementioned family of 
processes arises from  several scalings of discrete 
models in i.i.d. environments (see for instance \cite{BPS, bansima, BH})
and can be characterized by a stochastic differential equation (SDE) where the linear term is driven by a L\'evy process. 
Such a L\'evy process 
captures the effect of the environment on the mean offspring distribution of individuals. 
A process in this subclass is known as 
 \emph{CSBP in L\'evy environment} and  its construction has been given    by He et al. 
\cite{he2016continuous} and by Palau and Pardo \cite{PP1}, independently, as the unique strong solution of a SDE which will 
be specified below. 

The study of the long term behaviour  of CSBPs in L\'evy environment has attracted considerable attention recently, see for instance 
\cite{BPS,BH,li2016asymptotic,palau2017continuous,palau2016asymptotic}. In all these manuscripts, the speed of extinction has been computed for the case 
where the associated branching mechanism is  stable  since the  survival probability can be expressed explicitly in terms of  exponential functionals of  L\'evy processes.
In \cite{BH}, B\"oinghoff and Hutzenthaler have studied the Feller diffusion case in a Brownian environment and exploited the explicit density of the exponential functional of a Brownian motion with drift. Then
 Bansaye et al. 
\cite{BPS} studied the long term behaviour for stable branching mechanisms where the 
random environment is driven by a L\'evy process with bounded variation paths.  Palau et al. 
\cite{palau2016asymptotic} and Li and Xu \cite{li2016asymptotic}   extended these results and 
obtained  the extinction probability for  stable CSBPs in a general L\'evy environment.

Our aim is to relax the  assumption that the branching mechanism is stable, that is to say, we  are interested in studying the survival 
probability for a larger class of branching mechanisms associated to CSBPs in L\'evy environments. Here we focus on the critical case, more precisely in 
oscillating L\'evy environments satisfying the so-called  Spitzer's condition which is a well-known assumption in fluctuation theory of L\'evy processes (see assumption {\bf (H1)} below). In order to do so, we use two main tools in our arguments: fluctuation theory and the asymptotic behaviour of exponential functionals of L\'evy processes satisfying Spitzer's condition. We follow  the point of view  of Afanasyev et al. \cite{Afanasyev2005} in the discrete time setting, to deduce pathwise relationships between the dynamics of the CSBP in random environment and the L\'evy process driving 
the random environment on the survival event.
 More precisely, we prove that the survival of the process is  strongly related to its survival up to  the time when the random 
environment reaches its running infimum. Then, we  decompose its paths into two parts, the pre-infimum and post-infimum processes. If the process survives until the time when the random environment reaches its  running infimum, then it has a positive probability to survive after this time and, consequently, it  evolves in a 
``favorable'' environment.
As we will see below, the global picture stays unchanged compared to \cite{Afanasyev2005} but new difficulties arise in the continuous space setting. 
In particular,  the state $0$ can be polar and  
the process might  become very close to $0$ but never reach this point.  To focus on the absorption event, we use Grey's condition which guarantees  that $0$ is accessible.  Another difficulty arises at the  upper bound for the probability of survival.  Indeed, in the discrete setting, to bound the probability of survival we can use the fact that the probability that the process survives at times when the environment reaches a local minimum is equal to the probability that the current population size is bigger or equal than 1 at times when the environment reaches a local minimum. Then Chebyshev or Markov inequality will help to obtain a suitable  upper bound. In the continuous setting, this  strategy is not helpful. In fact, it is suitable to perform good estimates for the probability that the process survives at times when the environment reaches a local minimum. In order to do so
 explicit knowledge of the probability of extinction is required but the latter can only be derived in  few cases, even in the case when the environment is fixed. When the environment is fixed, good estimates of such probability  can be derived when the branching mechanism is regularly varying at $\infty$  with index  $\vartheta\in(1,2)$ or possesses a Blumenthal-Getoor index bigger than one. In our case, the latter type of estimates cannot be  obtained due to the environment as we will explain below. So in order to overcome these difficulties, we impose some assumptions on the branching mechanism and on the environment which are not so restrictive.  
\subsection{CSBPs in a L\'evy environment}
Let $(\Omega^{(b)}, \mathcal{F}^{(b)}, (\mathcal{F}^{(b)}_t)_{t\ge 0}, \P^{(b)})$ be a filtered probability space satisfying the usual hypothesis
and introduce a $(\mathcal{F}^{(b)}_t)_{t\ge 0}$-adapted standard Brownian motion  $B^{(b)}=(B^{(b)}_t,t\geq 0)$ and an independent $(\mathcal{F}^{(b)}_t)_{t\ge 0}$-adapted Poisson random measure 
$N^{(b)}(\ud s,\ud z,\ud u)$ defined on $\mathbb{R}^3_+$,
with intensity $\ud s\mu(\ud z)\ud u$. The measure $\mu$ is concentrated on $(0,\infty)$ and in the whole paper we assume that
\begin{equation}\label{finitemom}
\int_{(0,\infty)} (z\wedge z^2)\mu(\ud z)<\infty,
\end{equation}
which guarantees non-explosivity (see Lemma \ref{conservative} in the Appendix for the proof of this fact). 
We denote by $\widetilde{N}^{(b)}$ the compensated measure of $N^{(b)}$.

According to Dawson and Li \cite{dawsonLi06}, we  can define $Y=(Y_t, t\ge 0)$,  a CSBP,  as the unique strong solution of the 
following SDE
\begin{equation*}\label{csbp}
 Y_t=Y_0-\psi^\prime(0+)\int_0^t Y_s \ud s+\int_0^t \sqrt{2\gamma^2 Y_s}\ud B^{(b)}_s 
+\int_0^t\int_{(0,\infty)}\int_0^{Y_{s-}}z\widetilde{N}^{(b)}(\ud s,\ud z,\ud u),
\end{equation*}
where $\gamma\ge 0$ and
 $\psi^\prime(0+)\in \R$, denotes the right derivative at $0$ of the so-called branching mechanism associated to $Y$ which satisfies 
 the celebrated L\'evy-Khintchine formula, i.e.
\begin{equation}\label{defpsi}
 \psi(\lambda)= \lambda \psi'(0+) + \gamma^2 \lambda^2 +\int_{(0,\infty)}\left( e^{-\lambda x}-1 + \lambda x  \right)\mu (\ud x), \qquad \textrm{for}\quad \lambda\ge 0.
\end{equation}

For the random environment, we consider $(\Omega^{(e)},\mathcal{F}^{(e)}, (\mathcal{F}^{(e)}_t)_{t\ge 0}, \P^{(e)})$ a filtered probability space satisfying 
the usual hypothesis and    a $(\mathcal{F}^{(e)}_t)$-L\'evy process $K=(K_t,t\ge 0)$ which is defined as follows
 \begin{equation*}\label{env}
 K_t=\alpha t+\sigma B^{(e)}_t+\int_0^t\int_{\R\setminus (-1,1)} (e^z-1) {N}^{(e)}(\ud s,\ud z )
 +\int_0^t\int_{(-1,1)} (e^z-1) \widetilde{N}^{(e)}(\ud s,\ud z ),
 \end{equation*}
 where $\alpha\in \mathbb{R}$,  $\sigma\geq 0$, $B^{(e)}=(B^{(e)}_t,t\geq 0)$ denotes a $(\mathcal{F}^{(e)}_t)_{t\ge 0}$-adapted 
 standard Brownian motion  and  $N^{(e)}(\ud s, \ud z)$ 
 is a $(\mathcal{F}^{(e)}_t)_{t\ge 0}$-adapted Poisson random measure on $\R_+\times \R $  with intensity $\ud s\pi(\ud y)$, which is 
 independent of $B^{(e)}$. The measure $\pi$ is 
 concentrated on $\R\setminus\{0\}$ and fulfills the following integral condition
 $$\int_{\R} (1\wedge z^2)\pi(\ud z)<\infty.$$
 
In our model, the population size has no impact on the evolution 
of the environment and we are considering independent processes for demography and environment. 
More precisely, we work now on the product space $(\Omega, \mathcal{F}, (\mathcal{F})_{t\ge 0}, \P)$, where $\Omega=\Omega^{(e)}\times \Omega^{(b)}$, 
$\mathcal{F}=\mathcal{F}^{(e)} \otimes \mathcal{F}^{(b)}$,  and $\mathcal{F}_t=\mathcal{F}^{(e)}_t \otimes \mathcal{F}^{(b)}_t$ for $t\ge 0$, $ \P= \P^{(e)} \otimes  \P^{(b)}$ and we make the direct extension 
of 
$B^{(b)}$, $N^{(b)}$, $B^{(e)}$,  $N^{(e)}$ and $K$ to $\Omega$ by projection respectively on $\Omega^{(b)}$ and $\Omega^{(e)}$.
In particular, 
$(B^{(e)}, N^{(e)})$ is independent of $(B^{(b)}, N^{(b)})$.

Letting $Z_0\in [0,\infty)$ a.s.,
 a CSBP in a L\'evy environment $Z$
can be defined as the unique non-negative strong solution of the following SDE,
\begin{equation}\label{csbplre}
\begin{split}
 Z_t=&Z_0-\psi^\prime(0+)\int_0^t Z_s \ud s+\int_0^t \sqrt{2\gamma^2 Z_s}\ud B^{(b)}_s \\
&\hspace{3cm}+\int_0^t\int_{(0,\infty)}\int_0^{Z_{s-}}z\widetilde{N}^{(b)}(\ud s,\ud z,\ud u)+\int_0^t  Z_{s-}\ud K_s.
\end{split}
\end{equation}
According to He et al. \cite{he2016continuous} and Palau and Pardo \cite{PP1},  pathwise uniqueness and strong existence hold 
for this SDE.
Actually, Palau and Pardo also considered the case when $\psi^\prime(0+)=-\infty$, and obtained that the corresponding SDE has 
a unique 
strong solution up to explosion and by convention here it is identically equal to $+\infty$ after the explosion time. 
It turns out that \eqref{finitemom} is a sufficient  condition to conclude that the process  $Z$ is conservative or in other 
words that it does not explode in finite time. The conservativeness  was first observed by Palau and Pardo in 
\cite{{palau2017continuous}} (see Proposition 1) in 
the case when the random environment is driven by a Brownian motion. A similar result also holds in our context:
if \eqref{finitemom} holds then
\[
\mathbb{P}_z(Z_t<\infty)=1 , \qquad \textrm{for any} \qquad t\ge 0,
\]
and any $z\ge 0$ where $\mathbb{P}_z$ denotes the law of the process $Z$ starting from $z\ge 0$. The proof
follows  from similar arguments as those used in \cite{{palau2017continuous}} and is deferred to the Appendix 
(see Lemma \ref{conservative}).

The analysis of the process $Z$ is deeply related to the behaviour and fluctuations of the L\'evy process $\overline{K}$, 
defined on the same filtration as $K$, which provides a quenched martingale. We set
 \begin{equation}\label{envir2}
 \overline{K}_t=\overline{\alpha}t+\sigma B^{(e)}_t+\int_0^t\int_{(-1,1)}z \widetilde{N}^{(e)}(\ud s,\ud z)+
 \int_0^t\int_{\R\setminus (-1,1)}z {N}^{(e)}(\ud s,\ud z),
 \end{equation} 
where 
\begin{equation}\label{defbeta} \overline{\alpha} :=\alpha -\psi^{\prime}(0+)- \frac{\sigma^2}{2} - \int_{(-1,1)}(e^z-1-z)
\pi(\ud z), \end{equation}
and we obtain the following statement.
\begin{proposition}
\label{martingquenched}
 For $\P^{(e)}$ almost every $w^{(e)}\in\Omega^{(e)}$,  $\left(\exp\{-\overline{K}_t(w^{(e)},.)\}Z_t(w^{(e)},.): t\geq 0\right)$ 
 is 
 a $(\Omega^{(b)}, {\mathcal F}^{(b)}, \P^{(b)})$-martingale
and  for any $t\geq 0$ and $z\geq 0$,
$$\E_z[Z_t \ \vert \ K]=ze^{\overline{K}_t}, \ \qquad \P \ \textrm{-a.s}.$$
\end{proposition}
The proof is deferred to the Appendix. In other words, the process $\overline{K}$
plays an analogous role as the random walk associated to the logarithm  of the offspring means in the discrete time framework
and  leads to the usual classification for the long time behaviour of branching processes. We say that the 
process $Z$ is subcritical,  critical or supercritical accordingly as  
$\overline{K}$ drifts to $-\infty$, oscillates or drifts to $+\infty$. We refer to
\cite{BPS,BH,palau2016asymptotic,li2016asymptotic} 
for asymptotic results under different regimes. We observe that in the critical case and contrary to the discrete framework, 
the process may oscillate
between $0$ and $\infty$ a.s., see for instance  \cite{BPS}.

\subsection{Properties of the L\'evy environment}  Recall that $\overline{K}=(\overline{K}_t, t\ge0)$ denotes the real valued L\'evy process 
defined in \eqref{envir2}. That is to say $\overline{K}$ has  stationary and independent increments with  c\`adl\`ag paths.  
For simplicity, we denote  by $\mathbb{P}^{(e)}_x$ (resp. $\mathbb{E}^{(e)}_x$) the probability (resp. expectation) associated to  the process $\overline{K}$ 
starting from $x\in \mathbb{R}$, and when $x=0$, we use the notation $\mathbb{P}^{(e)}$ for $\mathbb{P}^{(e)}_0$ (resp. $\mathbb{E}^{(e)}$ for $\mathbb{E}_0^{(e)}$), i.e. 
\[
\mathbb{P}^{(e)}_x(\overline{K}_t\in B)=\mathbb{P}^{(e)}(\overline{K}_t+x\in B),\qquad \textrm{for }\quad B\in\mathcal{B}(\mathbb{R}).
\] We assume in the sequel 
that $\overline{K}$ is not a compound Poisson process.\\

In what follows, we assume  a general condition which is known as   \emph{Spitzer's condition} in fluctuation theory of 
L\'evy processes, i.e. 
$$
{\bf (H1)} \hspace{2cm} \frac{1}{t} \int_0^t \mathbb{P}^{(e)}(\overline{K}_s \geq 0) \ud s\longrightarrow \rho \in (0,1),\qquad 
\textrm{as }\quad t\to\infty. \hspace{3cm}
$$
Spitzer's condition  implies that $\overline{K}$ oscillates and implicitly, from Proposition \ref{martingquenched},  that the process $Z$ is in  the 
critical regime. According to  Bertoin and Doney \cite{bertoin1997spitzer} condition {\bf (H1)} is equivalent to
\[
 \mathbb{P}^{(e)}(\overline{K}_t \geq 0) \longrightarrow \rho \in (0,1),\qquad \textrm{as }\quad t\to\infty.
 \]
 Spitzer's condition is a key condition to understand the tail distribution of first passage times (see \eqref{limitv} and \eqref{infimumspitzer} for instance).
 Notice that if Spitzer's condition holds and $\bar{K}$ has a finite variance, then necessarily $\rho=1/2$.
 Examples of L\'evy processes satisfying Spitzer's condition are the standard  Brownian motion,
and stable processes where $\rho\in(0,1)$ plays the role of the  positivity parameter. Furthermore, any symmetric  L\'evy process 
satisfies Spitzer's condition 
with $\rho=1/2$ and any L\'evy process in the domain of attraction of a stable process with positivity parameter $\rho\in (0,1)$ 
as $t\to \infty$, 
satisfies Spitzer's condition.

Our arguments on the survival event rely on the running infimum    of $\overline{K}$, here denoted by $I=(I_t,t \geq 0)$ where
\begin{equation} \label{running_inf}
I_t=\inf_{0\le s\le t} \overline{K}_s, \qquad t\ge 0.
\end{equation}
To be more precise, we use fluctuation theory of L\'evy processes reflected at their running infimum. Let us recall that the 
reflected process $\overline{K}-I$  is a Markov process with respect to the  filtration $(\mathcal{F}^{(e)}_t)_{t \geq 0}$ and whose semigroup 
satisfies the Feller property (see for instance Proposition VI.1 in the monograph of Bertoin \cite{bertoin1998levy}).  We 
denote by $\widehat{L}$   the local time of $\overline{K}-I$ at $0$ in the sense of Chapter IV in \cite{bertoin1998levy}.
Similarly to the discrete framework \cite{Afanasyev2005}, the asymptotic analysis and the role of the initial condition 
involve the renewal function $\widehat{V}$ which is defined, for all $x\ge 0$, as follows
\begin{equation}\label{renewalfct}
 \widehat{V}(x):=\mathbb{E}^{(e)}\left[\int_{[0,\infty)}\mathbf{1}_{\{I_t\ge -x\}}\ud \widehat{L}_t\right].
\end{equation}
The renewal function $ \widehat{V}$ is subadditive, continuous and increasing and satisfies $ \widehat{V}(x)\ge  0$ for $x\ge 0$ and $\widehat{V}(x)>0$ for 
$x>0$. 
See for instance the monograph of Doney \cite{doney2007fluctuation} or Section 2  for further details about the previous facts.
 Under Spitzer's condition (see Theorem VI.18 in Bertoin \cite{bertoin1998levy}) the 
  asymptotic behaviour of the probability that the L\'evy process $\overline{K}$   remains positive, i.e. $\mathbb{P}^{(e)}_x(I_t>0)$ 
  for $x>0$, is regularly varying at $\infty$ with index $\rho-1$ and moreover, for any $x,y>0$, we have
\begin{equation}\label{limitv}
\lim_{t\to \infty}\frac{\mathbb{P}^{(e)}_x(I_t>0)}{\mathbb{P}^{(e)}_y(I_t>0)}=\frac{\widehat{V}(x)}{ \widehat{V}(y)}.
\end{equation}
In other words, we obtain that for any  $x >0$, 
 \begin{equation}\label{infimumspitzer}
 \mathbb{P}^{(e)}_x(I_t>0)\sim   \widehat{V}(x)t^{\rho-1}\ell (t), \quad \text{as} \quad t \to \infty, 
 \end{equation}
where  $\ell$ is a slowly varying function at $\infty$, that is to say, for $c>0$, 
\[
\lim_{t\to\infty}\frac{\ell(c t)}{\ell(t)}=1.
\]

\subsection{Main result}
We  now state our main result which is devoted to   the speed of extinction of CSBPs in L\'evy environment  under the assumption that the environment oscillates. 
It is important to note that the survival of the process is associated  to  "favorable" environments 
which are characterized by the running infimum of the environment,  which is not too small from our assumptions. 

We need some assumptions to control the effect of the environment on the event of survival.
The following moment assumption is needed to guarantee the non-extinction of the process in favorable environments 
(see Proposition \ref{posi_CSBP} for further details), 
$${\bf (H2)} \qquad \qquad  \qquad  \qquad \qquad \quad \int^\infty \theta \ln^2(\theta) \mu(\ud \theta) <\infty.  \qquad  \qquad \qquad  \qquad  \qquad \qquad \quad  \text{} $$
The above condition is similar to $x\log (x)$ moment condition on the measure $\mu$, used in Proposition 2 in \cite{palau2017continuous} to determine that  the probability of survival of CSBP processes in L\'evy environments that drifts to $+\infty$, is positive. 

As we will see below, Spitzer's condition ({\bf H1}) and  assumption ({\bf H2})  are sufficient conditions to provide a lower bound for the survival probability. In order to get the upper bound, further assumptions on the branching mechanism and the environment are required.  Let
\begin{equation}\label{psi0}
\psi_0(\lambda):=\psi(\lambda)-\lambda\psi'(0+),\qquad \textrm{for}\quad \lambda\ge 0,
\end{equation}
and assume that there exist $\beta\in(0,1]$,  $\theta^+>1$  and  $C>0$ such that
$$  \hspace{-2cm}{\bf (H3)}  \hspace{1.5cm}   \psi_0(\lambda)\geq C \lambda^{1+\beta}, \quad  \text{for } \lambda \geq 0.$$
Assumption {\bf (H3)} allows us to control the absorption of the process for bad environments (see Lemma \ref{L2preuve}) and 
in particular, it guarantees that
$\psi_0(\lambda)$  satisfies  the so-called Grey's condition, i.e.  
\begin{equation}\label{GreysCond} 
\int^\infty \frac{\ud z}{\psi_0(z)}<\infty,
 \end{equation}
which is a necessary and sufficient condition for absorption of CSBPs,
see for instance  \cite{MR0408016}.  Recently, He et al. \cite{he2016continuous} have shown that this condition is also  necessary and sufficient 
for CSBPs in a L\'evy environment to get absorbed with positive probability (see Theorem 4.1 in \cite{he2016continuous}).  In our case since the process $\overline{K}$ oscillates  and   Grey's condition \eqref{GreysCond} is satisfied  then absorption occurs a.s.  according to Corollary 4.4 in \cite{he2016continuous}.

\begin{theorem}\label{maintheo} Assume  that assumptions ${\bf (H1)-(H3)}$ hold. Then there exists a positive function $c$ such that for any  $z>0$,
 $$ \P_z(Z_t>0) \sim c(z) \mathbb{P}^{(e)}_1(I_t>0)\sim c(z) \widehat{V}(1) t^{\rho-1}\ell (t), \quad \text{as} \quad t \to \infty, $$
 where 
  $\ell$ is the slowly varying function introduced in \eqref{infimumspitzer}.
\end{theorem}
We point out that we only need assmptions ${\bf (H1)}$ and ${\bf (H2)}$ to ensure that the probability of survival 
of the process $Z$ satisfies 
\begin{equation}\label{lowbound} \P_z(Z_t>0) \geq c(z) \mathbb{P}^{(e)}_1(I_t>0)\sim  \widehat{V}(1) c(z)t^{\rho-1}\ell (t), \quad \text{as} \quad t \to \infty.
\end{equation}

It seems quite difficult to deduce the asymptotic behaviour of the probability of survival just under assumptions {\bf (H1)} and {\bf (H2)}, as we explain below. 
 Let us briefly explain  why stronger assumptions such as ${\bf (H3)}$ are required for the upper bound. Recall from   Proposition 2 in \cite{PP1}, that  there exists a functional $v_t(s,\lambda, \overline{K})$ which is  the  $\mathbb{P}^{(e)}$-a.s. unique solution of  the backward differential equation 
\begin{equation}\label{backward}
\frac{\partial}{\partial s} v_t(s, \lambda, \overline{K})=e^{\overline{K}_s}\psi_0(v_t(s, \lambda, \overline{K})e^{-\overline{K}_s}), 
\qquad v_t(t, \lambda, \overline{K})=\lambda,
\end{equation}
where $\psi_0$ is defined as in \eqref{psi0}.  For simplicity of exposition, we denote by $\P_{(z,x)}$ (resp. $\E_{(z,x)}$ its expectation) for the law of the couple $(Z,\overline{K})$ started from $(z,x)$ where $z,x>0$,  under $\mathbb{P}$. Thus, the functional  $v_t(s,\lambda, \overline{K})$ determines the law of the reweighted  process $(Z_t e^{-\overline{K}_t}, t\ge 0)$ as follows,
\[
\begin{split}
\mathbb{E}_{(z, 1)}\Big[\exp\Big\{-\lambda Z_t e^{-\overline{K}_t}\Big\}\Big]&=\mathbb{E}_{(z,0)}\Big[\exp\Big\{-\lambda e^{-1} Z_t e^{-\overline{K}_t}\Big\}\Big]\\
&=\mathbb{E}^{(e)}\Big[\exp\Big\{-zv_t(0,\lambda e^{-1},\overline{K})\Big\}\Big].
\end{split}
\]
Under  Grey's condition \eqref{GreysCond}  and the previous identity, we can deduce  
\begin{equation}\label{cotasup}
\P_{(z, 1)}(Z_t>0, I_t\le -y)=\mathbb{E}^{(e)}\left[\left(1-e^{-z \overline{v}_t(0,\infty, \overline{K})}\right)\mathbf{1}_{\{I_t\le -y-1\}}\right],  \qquad \textrm{ for } \quad y\ge 0, 
\end{equation}
where $\overline{v}_t(0,\infty, \overline{K})$ is $\mathbb{P}^{(e)}$-a.s. finite for all $t\ge 0$, (according to Theorem 4.1  in \cite{he2016continuous}) but perhaps equals 0. Actually, assumption {\bf (H2)} guarantees that $\overline{v}_t(0,\infty, \overline{K})>0$,  $\mathbb{P}^{(e)}$-a.s.,  for all $t>0$; and  even in ``favorable'' environments  (see Proposition \ref{posi_CSBP}). The right-hand side of \eqref{cotasup} seems difficult to estimate due to the nature of the functional $\overline{v}_t(0,\infty, \overline{K})$. Even under the assumption that $\psi_0$ is regularly varying at $\infty$, it is not so clear how to handle $\overline{v}_t(0,\infty, \overline{K})$ due to the dependence on the environment. In the discrete setting such a probability can be estimated in terms of the infimum of the environment (which is a random walk) since the event of survival is equal to the event of the current population being bigger or equal to one. Assumption {\bf (H3)} allows us to upper bound \eqref{cotasup} by the exponential functional of $\overline{K}$, and to study its asymptotic behaviour.

We end our exposition with some examples where the renewal function can be computed explicitly and the previous results can be 
applied.

\subsection{Examples}
 
{\it a) Brownian case.} In the particular case when $\overline{K}$ is a standard Brownian motion starting from $x>0$,  it is 
known that the renewal measure  is proportional to the identity, i.e. $ \widehat{V}(y)\propto y$, for $y\ge 0$.  Moreover,  Brownian motion 
oscillates and satisfies  Spitzer's condition 
 {\bf (H1)} with $\rho=1/2$. Then, assuming that conditions {\bf (H2)} and {\bf (H3)} are fulfilled,  we obtain that the  CSBP in a 
Brownian environment satisfies
\[
 \P_z(Z_t>0) \sim c(z) t^{-1/2}\ell (t), \quad \text{as} \quad t \to \infty. 
\]
In this particular case, we can compute  the function $\ell$, i.e.
\[
\ell(t)=\int_1^\infty e^{-\frac{1}{2tu}}\frac{\ud u}{\sqrt{2\pi u^3}},\qquad t> 0,
\]
which follows from the fact that the law of the infimum of a Brownian motion is given  by
\[
\mathbb{P}^{(e)}_1(I_t>0)=\int_t^\infty e^{-\frac{1}{2 w}}\frac{\ud w}{\sqrt{2\pi w^3}}, \qquad t>0.
\]

{\it b) Spectrally negative case}. If $\overline{K}$ is a  spectrally negative L\'evy process, i.e. it has no positive jumps, then the renewal measure is given by the so-called scale function $W:[0,\infty)\to[0,\infty)$, which is defined as the unique  continuous increasing function  satisfying
\[
\int_0^\infty e^{-\lambda x}W(x)\ud x=\frac{1}{\phi(\lambda)} \qquad \textrm{for} \qquad \lambda \ge 0,
\]
where $\phi$ denotes the Laplace exponent of $\overline{K}$ which is given by $\phi(\lambda):=\log\mathbf{E}[e^{\lambda \overline{K}_1}]$ and satisfies the so-called L\'evy-Khintchine formula. In other words, we identify the renewal function $\widehat{V}$ with the scale function W (i.e. $\widehat{V}\equiv W$).

 In this case, there is an interpretation of  Spitzer's condition in terms of the Laplace exponent $\phi$. More precisely, from Proposition VII.6 in Bertoin 
 \cite{bertoin1998levy}, the spectrally negative L\'evy process  $\overline{K}$ satisfies Spitzer's condition with $\rho\in(0,1)$ if and only if its 
 Laplace exponent $\phi$ is regularly varying at $0$ with index $1/\rho$. 
 This proposition also mentions that in this case, $\rho$ is necessarily larger than $1/2$.
Hence assuming that 
   the Laplace exponent $\phi$ is regularly varying at $0$ with index $1/\rho$,  Theorem 1 holds
under the assumption that the branching mechanism satisfies $\psi_0(\lambda)\ge C \lambda^{1+\beta}, $ for some $\beta>0$, together with condition {\bf (H2)}.

In the particular case where $\overline{K}$ is a spectrally negative stable process with index $\alpha\in (1,2)$, we have $W(x)=x^{\alpha-1}/\Gamma(\alpha)$, for $x\ge 0$, where $\Gamma$ denotes the so-called Gamma function. Moreover, it satisfies Spitzer's condition with $\rho=1/\alpha$. \\

{\it c) Stable case.}  Assume that $\alpha\in(1,2)$ and that $\overline{K}$ is a stable L\'evy process with positivity parameter $\rho\in(0,1)$. 
It is known that the descending ladder height is a stable subordinator with parameter $\alpha(1-\rho)$ (see for instance Lemma VIII.1 in \cite{bertoin1998levy} and Section 2 for a proper definition of the descending ladder height) which implies that  the renewal function $ \widehat{V}(x)$ is proportional to  $x^{\alpha(1-\rho)}$, for $x>0$. Indeed, its Laplace transform satisfies
\[
\int_0^\infty e^{-\lambda x}\widehat{V}(\ud x)=\frac{1}{\lambda^{\alpha(1-\rho)}}, \quad \lambda>0.
\]
Hence, Theorem 1 holds
under the assumption that the branching mechanism satisfies $\psi_0(\lambda)\ge C \lambda^{1+\beta}, $ for some $\beta>0$, together with condition {\bf (H2)}.\\

The remainder of the manuscript is  organized as follows. In Section 2, some preliminaries on fluctuation theory of 
L\'evy processes are introduced, as well as the definition of their conditioned version to stay positive. 
Moreover some useful properties of the latter are also studied. Section 3 is devoted to the study of CSBPs in a conditioned 
random environment whose properties are needed for our purposes. The proof of the main result is provided in Section 4 and, finally, in 
Appendix \ref{appendix}
we provide the proofs of Proposition \ref{martingquenched} as well as the non-explosivity of CSBPs in a L\'evy random environment.

\section{Preliminaries} \label{sectionLevyCondi}
In order to provide a precise description of the relationship between the survival probability of the process $Z$ and the behaviour of the running 
infimum of $\overline{K}$, 
the description of the  L\'evy process $\overline{K}$ conditioned to stay positive is needed  as well as the description of the process $Z$ under this conditioned 
random environment.

In this section, we introduce  L\'evy processes conditioned to stay positive as well as some of their properties that we will use in the sequel.

 \subsection{L\'evy processes and fluctuation theory}
Recall that $I_t=\inf_{0\le s\le t}\overline{K}_s$, for $t\ge 0$,   and that the reflected process $\overline{K}-I$ is a Markov process with respect to the  filtration $ (\mathcal{F}^{(e)})_{t\ge 0}$ and whose semigroup 
satisfies the Feller property. It is important to note that the same properties are satisfied by the  reflected process at its running supremum $S-\overline{K}$, where $S_t=\inf_{0\le s\le t}\overline{K}_s$, since the dual process $-\overline{K}$ is also a L\'evy process satisfying that for any fixed time $t>0$, the processes \[
(\overline{K}_{(t-s)}-\overline{K}_{t}, 0\le s\le t)\qquad  \textrm{and}\qquad (-\overline{K}_s, 0\le s\le t),
\]
 have the same law.

We also recall that $\widehat{L}$ denotes the local time of the reflected process $\overline{K}-I$ at $0$ in the sense of Chapter IV in \cite{bertoin1998levy}. Similarly, we denote by $L$ for the local time at 0 of $S-\overline{K}$.  If $0$ is regular for $(-\infty,0)$ or regular downwards, i.e.
\[
\mathbb{P}^{(e)}(\tau^-_0=0)=1,
\] 
where $\tau^{-}_0=\inf\{s> 0:  \overline{K}_s\le 0\}$, then $0$ is regular for the reflected process $\overline{K}-I$ and then, up to a multiplicative constant, $\widehat{L}$ is the unique additive functional of the reflected process whose set of increasing points is $\{t:\overline{K}_t=I_t\}$. If $0$ is not regular downwards then the set $\{t:\overline{K}_t=I_t\}$ is discrete and we define the local time $\widehat{L}$ as the counting process of this set.

Let us denote by $\widehat{L}^{-1}$ for the inverse local time and  introduce  the so-called descending ladder height process $\widehat{H}$ which is  defined  by
\begin{equation}\label{defwidehatH}
\widehat{H}_t=-I_{\widehat{L}_t^{-1}}, \qquad t\ge 0.
\end{equation}
The pair $(\widehat{L}^{-1}, \widehat{H})$ is a bivariate subordinator, as is $({L}^{-1}, {H})$ where
\[
{H}_t=S_{{L}_t^{-1}}, \qquad t\ge 0.
\]
Both pairs are known as descending and ascending ladder processes, respectively.  The  Laplace transform of the descending ladder process $(\widehat{L}^{-1}, \widehat{H})$ is such that for $ \theta,\lambda \geq 0$, 
\begin{equation}
 \label{sub}
 \mathbb{E}^{(e)}\left[\exp\left\{-\theta \widehat{L}^{-1}_t-\lambda\widehat{H}_t \right\}\right]=\exp\left\{-t \widehat{\kappa}(\theta, \lambda)\right\},\qquad  t\ge 0, 
 \end{equation}
 writing $\widehat{\kappa}(\cdot,\cdot)$ for its bivariate Laplace exponent ($\kappa(\cdot, \cdot)$ for  that of the 
 ascending ladder process) which has the form
 \[
 \widehat{\kappa}(\theta,\lambda)=\widehat{\delta}\theta+\widehat{\mathbf{d}}\lambda+\int_{(0,\infty)^2}\Big(1-e^{-(\theta x+\lambda y)}\Big)\widehat{\mu}(\ud x, \ud y),
 \]
with $\widehat{\delta}, \widehat{\mathbf{d}}\ge 0$ and 
\[
\int_{(0,\infty)^2}(x\land 1)(y\land 1)\widehat{\mu}(\ud x, \ud y)<\infty.
\]
Notice that both $(\widehat{L}^{-1}, \widehat{H})$ and $({L}^{-1}, {H})$ have no killing terms, since we are assuming that 
the process $\overline{K}$ oscillates.
Implicity, the Laplace exponent of $\widehat{H}$  satisfies
\begin{equation}\label{defwidehatkappa}
\widehat{\kappa}(0,\lambda)=\widehat{\mathbf{d}}\lambda +\int_{(0,\infty)} \Big(1-e^{-\lambda y}\Big)\widehat{\eta}(\ud y),
\end{equation}
where  $\widehat{\eta}(B)=\widehat{\mu}((0,\infty), B)$ for $B\in\mathcal{B}((0,\infty))$.

An interesting connection between the distribution of the ladder processes and that of $\overline{K}$ is given by the 
Wiener-Hopf factorisation
\begin{equation}\label{WHfactors}
\mathbb{E}^{(e)}\Big[e^{i\theta \overline{K}_{\mathbf{e}_q}}\Big]=\mathbb{E}^{(e)}\Big[e^{i\theta S_{\mathbf{e}_q}}\Big]\mathbb{E}^{(e)}\Big[e^{i\theta I_{\mathbf{e}_q}}\Big], 
\end{equation}
where $\mathbf{e}_q$ denotes an exponential random variable with parameter $q\ge 0$ which is independent of $\overline{K}$,
\[
\mathbb{E}^{(e)}\Big[e^{i\theta S_{\mathbf{e}_q}}\Big]=\frac{\kappa(q, 0)}{\kappa(q, -i\theta)} \qquad \textrm{and} \qquad\mathbb{E}^{(e)}\Big[e^{i\theta I_{\mathbf{e}_q}}\Big]=\frac{\widehat{\kappa}(q, 0)}{\widehat{\kappa}(q, i\theta)}.
\]
 We refer to Chapter VI in Bertoin \cite{bertoin1998levy} or Chapter 4 in of Doney \cite{doney2007fluctuation} for further details on the descending ladder processes $(\widehat{H}, \widehat{L})$ as well as for the Wiener-Hopf factorisation.

Next, we consider 
the renewal function $\widehat{V}$ which was defined in \eqref{renewalfct}.
It is known that $\widehat{V}$ is a finite, continuous, increasing  and subadditive function on $[0,\infty)$ satisfying 
\begin{equation}
\label{grandO}
 \widehat{V}(x)\leq C_1 x,\quad \text{  for any } \quad x\geq 0,
\end{equation}
where $C_1$ is a finite constant  (see for instance  Lemma 6.4 and Section 8.2 in the monograph of Doney 
\cite{doney2007fluctuation}). Moreover 
$\widehat{V}(0)=0$ if $0$ is regular downwards 
and $ \widehat{V}(0)=1$ otherwise. By a simple change of variables we can relate the definition of the renewal function $\widehat{V}$ and the descending ladder height $\widehat{H}$.
Indeed,  the measure induced by $\widehat{V}$ can be rewritten as follows,
\[
\widehat{V}(\ud x)=\mathbb{E}^{(e)}\left[\int_{0}^\infty\mathbf{1}_{\{\widehat{H}_t \in \ud x\}}\ud t\right].
\]
 Roughly speaking, the renewal function $ \widehat{V}(x)$ ``measures'' the amount of time that the descending ladder height process spends on the interval $[0,x]$ 
 and in particular induces a measure on $[0,\infty)$ which is known as the renewal measure. The latter implies
 \begin{equation}\label{LaplaceV}
 \int_{[0,\infty)}e^{-\lambda x}{\widehat{V}}(\ud x)=\int_{0}^\infty \mathbb{E}^{(e)}\left[e^{-\lambda \widehat{H}_t } \right]\ud t=\frac{1}{\widehat{\kappa}(0,\lambda)}, 
 \qquad \textrm{for}\quad \lambda \ge 0. 
 \end{equation}

Similarly, we introduce the  renewal funtion $ V$ associated with the ascending ladder height induced by
\begin{equation}\label{ascren}
 V(\ud x)=\mathbb{E}^{(e)}\left[\int_{0}^\infty\mathbf{1}_{\{{H}_t \in \ud x\}}\ud t\right],
\end{equation}
 which is also a finite, continuous, increasing  and subadditive function on $[0,\infty)$ such that $V(0)=0$ if $0$ is regular upwards 
and $ V(0)=1$ otherwise.
\subsection{L\'evy processes conditioned to stay positive}
Let us define the probability $\mathbb{Q}_{x}$ associated  to the L\'evy process $\overline{K}$ started at $x>0$ and killed  at time $\zeta$ when it 
first enters $(-\infty, 0)$, that is to say
$$ \mathbb{Q}_{x}\big[f(\overline{K}_t)\mathbf{1}_{\{\zeta>t\}}\Big ]:= \mathbb{E}^{(e)}_{x}\Big[f(\overline{K}_t)\mathbf{1}_{\{I_t> 0\}}\Big], $$
where $f:\mathbb{R}_+\to \mathbb{R}$ is measurable.

According to Lemma 1 in Chaumont and Doney \cite{chaumont2005levy}, under the assumption that $\overline{K}$ does not drift towards $-\infty$,  we have that the renewal function $\widehat{V}$ is invariant for the killed process. In other words,   for all $x> 0$ and $t\ge 0$,
\begin{equation}
\label{fctharm}
\mathbb{Q}_x\left[\widehat{V}(\overline{K}_t)\mathbf{1}_{\{\zeta>t\}}\right]=\mathbb{E}^{(e)}_x\left[ \widehat{V}(\overline{K}_t)\mathbf{1}_{\{I_t> 0\}}\right]=\widehat{V}(x).
\end{equation}

We now recall the definition of L\'evy processes conditioned to stay positive   as a Doob-$h$ transform. Before doing so, 
let us recall that   $\overline{K}$ is adapted to the filtration $(\mathcal{F}^{(e)}_t)_{t\ge 0}$.
Under the assumption that $\overline{K}$  does not drift towards $-\infty$, the law of the process $\overline{K}$ conditioned to stay positive is defined as follows, for $\Lambda \in \mathcal{F}^{(e)}_t$ and $x>0$, 
\begin{equation} \label{defPuparrowx}
\mathbb{P}^{(e),\uparrow}_{x} (\Lambda)
:=\frac{1}{ \widehat{V}(x)}\mathbb{E}^{(e)}_{x}\left[ \widehat{V}(\overline{K}_t) \mathbf{1}_{\{I_t >  0\}} \mathbf{1}_{\Lambda}\right].
\end{equation}

The term \textit{conditioned to stay positive} in definition \eqref{defPuparrowx} is justified from the following convergence 
result due to Chaumont \cite{chaumont1996conditionings} (see also Remark 1 in the aforementioned paper as well as
Chaumont and Doney \cite{chaumont2005levy}) that we recall here in the particular case when the process $\overline{K}$ fulfills 
Spitzer's condition  {\bf (H1)}.     \begin{lemma}\label{thmcondpos}
  Assume  that Spitzer's condition {\bf (H1)} is fulfilled.  Then,  for all $x>0$, $t \geq 0$ and $\Lambda \in \mathcal{F}^{(e)}_t$, 
  $$ \lim_{s \to \infty} \mathbb{P}^{(e)}_x(\Lambda | \overline{K}_u>0, 0 \leq u \leq s)=\mathbb{P}_x^{(e),\uparrow}(\Lambda).  $$
 \end{lemma}
The following inequality is also important for our purposes. Recall that $\hat{\kappa}$ denotes the Laplace exponent of the 
descending ladder process (see identity \eqref{sub}) and that $\tau_0^- = \inf \{ s \geq 0 : \bar{K}_s \leq 0 \}$.
\begin{lemma}
For $x>0$, we have
\begin{equation}\label{majP}
 \mathbb{P}^{(e)}_x(\tau^{-}_0 >t)\le 2e\widehat{\kappa}(1/t,0) \widehat{V}(x), \qquad \textrm{for}\quad t>0.
 \end{equation}
\end{lemma}
\begin{proof} 
We first observe that the following series of inequalities holds for $ q, t>0$, 
 \[
 \frac{t}{2}e^{-q t}\mathbb{P}^{(e)}_x(\tau^{-}_0 >t)\le \int_{t/2}^t e^{-qs}\mathbb{P}^{(e)}_x(\tau^{-}_0>s) \ud s\le \int_{0}^\infty e^{-qs}\mathbb{P}^{(e)}_x(\tau^{-}_0>s)\ud s.
 \]

From the Wiener-Hopf factorization \eqref{WHfactors}, we have
 \[
 \mathbb{E}^{(e)}\Big[e^{\theta I_{\mathbf{e}_q}}\Big]=\frac{\widehat{\kappa}(q, 0)}{\widehat{\kappa}(q, \theta)}, 
 \]
 where $\mathbf{e}_q$ is an exponential random variable with parameter $q>0$, which is independent of $\overline{K}$. Hence, by a classical  identity on tail distribution using  
 Fubini's theorem, we deduce
 \[
 \begin{split}
 \frac{\widehat{\kappa}(q, 0)}{\widehat{\kappa}(q, \theta)}&=\int_0^\infty e^{-\theta x}\mathbb{P}^{(e)}(-I_{\mathbf{e}_q}\in \ud x)\\
 & =\theta\int_0^\infty e^{-\theta x}\mathbb{P}^{(e)}(I_{\mathbf{e}_q}> -x)\ud x=\theta\int_0^\infty e^{-\theta x}\mathbb{P}^{(e)}_x(\tau^{-}_0> \mathbf{e}_q)\ud x.
 \end{split}
 \]
 Next, for every $q>0$, we consider the function given by
 \[
\widehat{V}^{(q)}(x):=\mathbb{E}^{(e)}\left[\int_0^\infty e^{-q \widehat{L}^{-1}_t} \mathbf{1}_{\{\widehat{H}_t\le x\}}\ud t\right].
 \]
 Performing a straightforward computation and using identity \eqref{sub}, we deduce
 \[
\theta \int_0^\infty e^{-\theta x}\widehat{V}^{(q)}(x)\ud x=\mathbb{E}^{(e)}\left[\int_0^\infty \exp\left\{-q \widehat{L}^{-1}_t-\theta \widehat{H}_t \right\}\ud t \right]=\frac{1}{\widehat{\kappa}(q, \theta)}.
 \] 
 The latter  implies
 \[
 q\int_0^\infty e^{-q s}\mathbb{P}^{(e)}_x(\tau^{-}_0>s)\ud s=\widehat{\kappa}(q,0)\widehat{V}^{(q)}(x).
 \]
We thus deduce for $t, q> 0$, that
 \[
 \frac{t}{2}e^{-q t}\mathbb{P}^{(e)}_x(\tau^{-}_0>t)\le \frac{\widehat{\kappa}(q,0)}{q}\widehat{V}^{(q)}(x)\le \frac{\widehat{\kappa}(q,0)}{q}\widehat{V}(x).
 \]
Taking $q=1/t$ yields \eqref{majP}, and completes the proof.
 \end{proof}

\section{CSBP  in a conditioned random environment} \label{SectionCSBPCondi}

\subsection{Definition and first properties}

Similarly to the definition of L\'evy processes conditioned to stay positive \cite{chaumont2005levy} and following a similar strategy as in the discrete framework in Afanasyev et al. \cite{Afanasyev2005}, 
we would like to  introduce a  CSBP in a L\'evy environment  conditioned to stay positive as a Doob-$h$ transform.  In order to do so, we first observe that $( \widehat{V}(\overline{K}_t) \mathbf{1}_{\{I_t \geq 0\}}, t\ge0)$
is also a martingale with respect to  $(\mathcal{F}_t)_{t\ge 0}$, under $\mathbb{P}$. This result is more or less clear since it is a martingale  under $\mathbb{P}^{(e)}$. Nonetheless we provide its  proof for the sake of completeness.

 Recall that   $\P_{(z,x)}$ (resp. $\E_{(z,x)}$ its expectation) denotes  the law of the couple $(Z,\overline{K})$ started from $(z,x)$ where $z,x>0$,  under $\mathbb{P}$.
\begin{lemma} Let us assume that $z,x>0$. The process $( \widehat{V}(\overline{K}_t) \mathbf{1}_{\{I_t \geq 0\}}, t\ge0)$ is a martingale with respect to $(\mathcal{F}_t)_{t\ge 0}$, under $\mathbb{P}_{(z,x)}$. 
\end{lemma}
\begin{proof} Let  $s\ge 0$ and $A \in \mathcal{F}_s$. We first claim  that $\mathbb{P}(A| K)$ is a  $\mathcal{F}^{(e)}_s$-measurable r.v. Indeed, since the family of sets
  \[
  \mathcal{C}_s=\{F_b\times F_e: F_b\in \mathcal{F}^{(b)}_s, F_e\in \mathcal{F}^{(e)}_s\}, 
  \] 
is a $\pi$-system that generates $\mathcal{F}_s$, we deduce that for any $D\in \mathcal{C}_s$ such that $D=B\times C$ with $B\in  \mathcal{F}^{(b)}_s$ and  $C\in\mathcal{F}^{(e)}_s$, the following identity holds
\[
\mathbb{P}(D|K)=\mathbf{1}_{C}\mathbb{P}(B|K)=\mathbf{1}_C\mathbb{P}^{(b)}(B),
\]
where in the last identity we have used that $B$ is independent of the environment and that $\mathbb{P}^{(b)}$ is the projection of $\mathbb{P}$ on $\Omega^{(b)}$.  
A monotone class argument allows us to conclude our claim.

Next, we assume  $s\le t$ and   take $A\in \mathcal{F}_s$. By conditioning on the environment  and recalling that $\mathbb{P}^{(e)}$ is the projection of $\mathbb{P}$ on $\Omega^{(e)}$, we observe 
\[
\begin{split}
\mathbb{E}_{(z,x)}\left[ \widehat{V}(\overline{K}_t) \mathbf{1}_{\{I_t >  0\}}\mathbf{1}_{A}\right]&=\mathbb{E}_{(z,x)}\left[\widehat{V}(\overline{K}_t) \mathbf{1}_{\{I_t >  0\}}\mathbb{P}_{(z, x)}(A|K)\right]\\
&=\mathbb{E}^{(e)}_{x}\left[ \widehat{V}(\overline{K}_t) \mathbf{1}_{\{I_t >  0\}}\mathbb{P}_{(z, x)}(A|K)\right]
\end{split}.
\]
Let us now introduce the process $\widetilde{K}$ via $\widetilde{K}_u:=\overline{K}_{u+s}-\overline{K}_s$, for $u\ge 0$, which is independent of $\mathcal{F}^{(e)}_s$ and has the same law as $\overline{K}$.  We also define   its running infimum up to time $t$ by $\widetilde{I}_t$, i.e.
\[
\widetilde{I}_t=\inf_{0\le u\le t}\widetilde{K}_u.
\]
By taking $\mathfrak{P}$ a $\mathcal{F}^{(e)}_s$-measurable random variable, we deduce by conditioning on  $\mathcal{F}^{(e)}_s$  and from identity \eqref{fctharm}, that
\[
\begin{split}
\mathbb{E}^{(e)}_{x}\left[ \widehat{V}(\overline{K}_t) \mathbf{1}_{\{I_t >  0\}}\mathfrak{P}\right]&=\mathbb{E}^{(e)}_{x}\left[ \widehat{V}(\overline{K}_s+{K}_{t-s}) \mathbf{1}_{\{\widetilde{I}_{t-s}+ \overline{K}_s >  0\}}\mathbf{1}_{\{I_s >  0\}}\mathfrak{P}\right]\\
&=\mathbb{E}^{(e)}_{x}\left[\mathfrak{P}\mathbf{1}_{\{I_s>  0\}}\mathbb{E}^{(e)}_{\overline{K}_s}\left[\widehat{V}(\widetilde{{K}}_{t-s}) \mathbf{1}_{\{\widetilde{I}_{t-s}>  0\}}\right]\right]\\
&=\mathbb{E}^{(e)}_{x}\left[\mathfrak{P}\mathbf{1}_{\{I_s >  0\}} \widehat{V}(\overline{K}_s)\right].
\end{split}
\]
Putting all pieces together, we obtain
\[
\begin{split}
\mathbb{E}_{(z,x)}\left[\widehat{V}(\overline{K}_t) \mathbf{1}_{\{I_t >  0\}}\mathbf{1}_{A}\right]&=\mathbb{E}^{(e)}_{x}\left[\widehat{V}(\overline{K}_s) \mathbf{1}_{\{I_s >  0\}}\mathbb{P}_{(z, x)}(A|K)\right]\\
&=\mathbb{E}_{(z,x)}\left[ \widehat{V}(\overline{K}_s) \mathbf{1}_{\{I_s >  0\}}\mathbb{P}_{(z, x)}(A|K)\right]\\
&=\mathbb{E}_{(z,x)}\left[ \widehat{V}(\overline{K}_s) \mathbf{1}_{\{I_s >  0\}}\mathbf{1}_A\right],
\end{split}
\]
which allows us to conclude that the process $( \widehat{V}(\overline{K}_t) \mathbf{1}_{\{I_t >  0\}}, t\ge0)$ is a martingale with respect to $(\mathcal{F}_t)_{t\ge 0}$, under $\mathbb{P}_{(z,x)}$.
\end{proof}

From the previous result, we construct the law of a CSBP in a L\'evy environment  conditioned to stay positive as a Doob-$h$ transform. To be more precise, for $\Lambda \in \mathcal{F}_t$ and $x,z>0$, we define
\begin{equation} \label{defPuparrowzx}
\mathbb{P}^{\uparrow}_{(z,x)} (\Lambda):=\frac{1}{ \widehat{V}(x)}\mathbb{E}_{(z,x)}\left[ \widehat{V}(\overline{K}_t) \mathbf{1}_{\{I_t >  0\}} \mathbf{1}_{\Lambda}\right].
\end{equation}

Similarly as in Lemma \ref{thmcondpos}, the term \textit{L\'evy environment conditioned to stay positive} in definition \eqref{defPuparrowzx} is justified from the following convergence result, which is crucial to prove  Theorem~\ref{maintheo}. 
  \begin{lemma} \label{lemme25AGKV}
 Assume that Spitzer's condition {\bf (H1)} holds and let $z,x>0$. For $t \geq 0$ and $\Lambda \in \mathcal{F}_t$,  we have
  $$ \lim_{s \to \infty} \mathbb{P}_{(z,x)}(\Lambda | \overline{K}_u>0, 0 \leq u \leq s)=\mathbb{P}_{(z,x)}^{\uparrow}(\Lambda).  $$
Moreover if    $(G_t, t \geq 0)$ is  a uniformly bounded process which is  adapted  to $(\mathcal{F}_t)_{t\ge 0}$ and such that it converges to  $G_\infty,$ as $t \to \infty$,
 $\mathbb{P}^{\uparrow}_{(z,x)} $-almost surely, 
 then
 $$ \lim_{t \to \infty}\mathbb{P}_{(z,x)} \big[ G_t \big| \overline{K}_u>0, 0 \leq u \leq t\big]=\mathbb{P}^{\uparrow}_{(z,x)} \big[G_\infty\big].  $$
 \end{lemma}
 \begin{proof} We proceed similarly as in Proposition 1 in \cite{chaumont2005levy}. 
Let $ h, t\ge 0$ and take  $\Lambda \in \mathcal{F}_t$. Then from the Markov property at time $t$, we obtain
 \begin{equation}\label{exparrow}
 \mathbb{P}_{(z,x)}( \Lambda |I_{t+h}> 0 ) =
 \mathbb{E}_{(z,x)} \left[ \mathbf{1}_\Lambda \frac{\mathbb{P}^{(e)}_{\overline{K}_t}(I_h>  0)}{\mathbb{P}^{(e)}_{x}(I_{t+h}>  0)}\mathbf{1}_{\{I_t > 0\}}\right].
\end{equation}
From inequality \eqref{majP}, we see
 \[
 \frac{\mathbb{P}^{(e)}_{\overline{K}_t}(I_h>  0)}{\mathbb{P}^{(e)}_{x}(I_{t+h}>  0)}\mathbf{1}_{\{I_t > 0\}}\leq  2e \frac{\widehat{\kappa}\Big(h^{-1}, 0\Big)}{\mathbb{P}^{(e)}_x(\tau^-_0> t+h)}  \widehat{V}(\overline{K}_t)\mathbf{1}_{\{I_t > 0\}}.
  \]
On the other hand from  Spitzer's condition, we know  that $\widehat{\kappa}(\cdot, 0)$ is regularly varying at $0+$ with index 
 $1-\rho$ and $ \mathbb{P}^{(e)}_x(\tau^-_0>\cdot)$ is also regularly varying with index $\rho-1$ at $\infty$. Moreover, there is a slowly varying function $\ell(\cdot)$ at $\infty$ such that 
  \begin{equation} \label{asymp_kappa_hat}
 \widehat{\kappa}(q,0) \sim\frac{\Gamma(1+\rho)}{\rho} \ell(1/q)q^{1-\rho}, \quad \text{as} \quad q \to 0, 
 \end{equation}
 and
 \[
  \mathbb{P}^{(e)}_x(\tau^-_0> t)\sim \widehat{V}(x)t^{\rho-1}\ell (t),\quad \text{as} \quad t \to \infty,
\]
 see for instance the proof of Theorem VI.18 in \cite{bertoin1998levy}.    Therefore from  Potter's Theorem (see Theorem 1.5.6 in Bingham et al. \cite{bingham1989regular}) for any 
$C_2>1$ and $\delta>0$ there exists $M$ such that for $h\ge M$,
 \[
 \frac{\mathbb{P}^{(e)}_{\overline{K}_t}(I_h>  0)}{\mathbb{P}^{(e)}_{x}(I_{t+h}>  0)}\mathbf{1}_{\{I_t > 0\}}\leq  C(\rho) \left(1+\frac{t}{M}\right)^{1-\rho+\delta}   \frac{\widehat{V}(\overline{K}_t)}{\widehat{V}(x)}\mathbf{1}_{\{I_t > 0\}},
  \]
with 
\begin{equation}\label{csterho}
C(\rho)=\frac{2e \Gamma(1+\rho)}{\rho} 
  C_2.
\end{equation}
Since $\mathbb{E}_{(z,x)}[\widehat{V}(\overline{K}_t)\mathbf{1}_{\{I_t > 0\}}]=\widehat{V}(x)$ and $\widehat{V}$ is finite,  we may apply Lebesgue's theorem of dominated convergence on the right-hand side of \eqref{exparrow} when $h$ goes to $\infty$.  We conclude from the asymptotic \eqref{limitv} and the definition of \eqref{defPuparrowzx}.
 
 For the second part of our statement, we use  similar arguments as those used in Lemma 2.5 in \cite{Afanasyev2005}.  We let  $s \leq t$ and $\gamma \in (1,2]$ and 
 apply the Markov property at time $t$ and inequality \eqref{majP}, to deduce that
 \begin{eqnarray*}
  \left|\mathbb{E}_{(z,x)}\Big[G_t - G_s \Big|I_{\gamma t}> 0 \Big]\right| & \leq &
 \mathbb{E}_{(z,x)} \left[ \Big|G_t - G_s\Big| \frac{\mathbb{P}^{(e)}_{\overline{K}_t}(I_{(\gamma-1)t}>  0)}{\mathbb{P}^{(e)}_{x}(I_{\gamma t}>  0)}\mathbf{1}_{\{I_t > 0\}}\right]\\
  & \leq & 2e \frac{\widehat{\kappa}\Big(\frac{1}{(\gamma-1)t}, 0\Big)}{\mathbb{P}^{(e)}_x(\tau^-_0>\gamma t)}
  \mathbb{E}_{(z,x)} \Big[ |G_t - G_s|  \widehat{V}(\overline{K}_t)\mathbf{1}_{\{I_t > 0\}}\Big]\\
  & = & 2e \frac{\widehat{\kappa}\Big(\frac{1}{(\gamma-1)t}, 0\Big) \widehat{V}(x)}{\mathbb{P}^{(e)}_x(\tau^-_0>\gamma t)}
  \mathbb{E}_{(z,x)}^\uparrow \Big[|G_t - G_s|\Big].
 \end{eqnarray*}
 Again Potter's Theorem (see Theorem 1.5.6 in Bingham et al. \cite{bingham1989regular}) guarantees that for any
$C_2>1$ and $\delta>0$ there exists $M$ such that for $t\ge M$,
 \[
 \begin{split}
 & \left|\mathbb{E}_{(z,x)}\Big[G_t - G_s \Big|I_{\gamma t}> 0 \Big]\right|\\
 &\hspace{3cm}\le C(\rho)\max\left\{\left( \frac{\gamma}{\gamma-1} \right)^{\delta+{1-\rho}}, \left( \frac{\gamma}{\gamma-1} \right)
  ^{-\delta+1-\rho}\right\}\mathbb{E}_{(z,x)}^\uparrow \Big[|G_t - G_s|\Big],
  \end{split}
  \]
 and $C(\rho)$ is defined in \eqref{csterho}.

 Let $\eps>0$. As $(G_t, t \geq 0)$ is  a uniformly bounded process which converges to  $G_\infty,$ as $t \to \infty$,
 $\mathbb{P}^{\uparrow}_{(z,x)} $-almost surely, there exists $A_\eps >0$ such that for any $A_\eps \leq s \leq t$,
  $$ C(\rho)\max\left\{\left( \frac{\gamma}{\gamma-1} \right)^{\delta+{1-\rho}}, \left( \frac{\gamma}{\gamma-1} \right)
  ^{-\delta+1-\rho}\right\}\mathbb{E}_{(z,x)}^\uparrow \Big[|G_t - G_s|\Big] \leq \eps$$ 
and 
\begin{equation} \label{uti_conv} \left| \mathbb{E}_{(z,x)}^\uparrow [G_s]-\mathbb{E}_{(z,x)}^\uparrow [G_\infty]  \right|\leq \eps. \end{equation}
 Hence for any $s \geq A_\eps$, letting $t$ go to infinity and applying the first statement of this lemma, we get
$$\limsup_{t \to \infty}\left| \frac{\mathbb{E}_{(z,x)}\left[G_t \mathbf{1}_{\{I_{\gamma t}>0 \}}\right]}{\mathbb{P}^{(e)}_{x}(I_{\gamma t}> 0)}-\mathbb{E}_{(z,x)}^\uparrow [G_s] \right|\leq \eps.$$
Adding \eqref{uti_conv}, we get 
 $$\mathbb{E}_{(z,x)}\left[G_t \mathbf{1}_{\{I_{\gamma t}>0 \}} \right]=\Big(\mathbb{E}_{(z,x)}^\uparrow[G_\infty]+o(1)\Big)\mathbb{P}^{(e)}_{x}(I_{\gamma t}> 0). $$
 Thus, there exists a   constant $C_3>0$ such that
 \[
 \begin{split}
 \Big|\mathbb{E}_{(z,x)}[G_t\mathbf{1}_{\{I_t > 0\}}]- \mathbb{E}_{(z,x)}^\uparrow[G_\infty]\mathbb{P}^{(e)}_{x}(I_{ t}> 0)\Big|&\le 
 C_3 \mathbb{P}^{(e)}_{x}(I_t > 0, I_{\gamma t} \le 0)\\
 &+ \Big|\mathbb{E}_{(z,x)}[G_t\mathbf{1}_{\{I_{\gamma t} >0\}}]- \mathbb{E}_{(z,x)}^\uparrow[G_\infty]\mathbb{P}^{(e)}_{x}(I_{\gamma t}> 0)\Big|\\
  &\leq \Big(o(1)+c(1-\gamma^{\rho -1})\Big)\mathbb{P}^{(e)}_{x}(I_t > 0),
 \end{split}
 \]
where we applied  the asymptotic in \eqref{infimumspitzer} for the second inequality. Note that since $\ell$ in  \eqref{infimumspitzer} is slowly varying and $\gamma \in (1,2]$, we can choose
$c$ independent from $\gamma$ (see again Theorem 1.5.6 in Bingham et al. \cite{bingham1989regular}).
 Since the choice of $\gamma$ on $(1,2]$ was arbitrary, we finally obtain
 $$\mathbb{E}_{(z,x)}[G_t\mathbf{1}_{\{I_t > 0\}}]- \mathbb{E}_{(z,x)}^\uparrow[G_\infty]\mathbb{P}^{(e)}_{x}(I_t> 0)= o(\mathbb{P}^{(e)}_{x}(I_t > 0)) ,$$
 which completes the proof.
 \end{proof}

 \subsection{Non-absorption}
In this section,  we are interested in the event of survival of the process $Z$  under the conditioned environment. To estimate the latter, we first compute the probability of the event of extinction at a given time, under the conditioned environment, and then we will observe that such a probability is strictly positive if and only if  Grey's condition \eqref{GreysCond} is fullfilled. It is important to note that the latter statement can be deduced directly from Theorem 4.1 in \cite{he2016continuous} but actually, in this case its proof is rather simple and for completeness we decide to include it.

Recall from   Proposition 2 in \cite{PP1} (or after the comments of Theorem 1), that  there exists a functional $v_t(s,\lambda, \overline{K})$ which is  the   $\mathbb{P}^{(e)}$-a.s. unique solution of  the backward differential equation given in \eqref{backward} and satisfies
\begin{equation}\label{defvtslambdaK}
\begin{split}
\mathbb{E}_{(z, 0)}\Big[\exp\Big\{-\lambda e^{-x} Z_t e^{-\overline{K}_t}\Big\}\Big|\mathcal{F}^{(e)}_t\Big]&=\exp\Big\{-zv_t(0,\lambda e^{-x},\overline{K})\Big\}.
\end{split}
\end{equation}
 A similar identity holds for  CSBPs in a L\'evy environment  conditioned to stay positive as we see below.
\begin{proposition}\label{posi_CSBP1} For $x,z> 0$ and $\lambda\ge 0$, we have
\[
\E_{(z,x)}^\uparrow \left[ e^{-\lambda Z_t e^{-\overline{K}_t}} \right] =\mathbb{E}^{(e), \uparrow}_x  \left[ e^{- zv_t(0,\lambda e^{-K_0},\overline{K}-K_0) } \right],
\]
In particular, 
$$ \P^\uparrow_{(z,x)} (Z_t=0) = \mathbb{E}^{(e), \uparrow}_x \left[ e^{-z v_t(0,\infty,\overline{K}-K_0) } \right], \qquad \textrm{for}\quad t> 0,$$
which is strictly positive if and only if  Grey's condition \eqref{GreysCond} is satisfied. 
\end{proposition}
\begin{proof} Let  $x,z> 0$. From the definition of  CSBPs in a L\'evy environment  conditioned to stay positive \eqref{defPuparrowzx}, we deduce that for every non-negative $\lambda$,
 \[
 \begin{split} \E_{(z,x)}^\uparrow \left[ e^{-\lambda Z_t e^{-\overline{K}_t}} \right] 
 &=\frac{1}{ \widehat{V}(x)}\E_{(z,x)} \left[ \widehat{V}(\overline{K}_t)e^{-\lambda Z_t e^{-\overline{K}_t}} \mathbf{1}_{\{I_t >  0\}}\right] 
\\ 
&=\frac{1}{ \widehat{V}(x)}\E_{(z,0)} \left[  \widehat{V}(\overline{K}_t+x)e^{-\lambda e^{-x} Z_t e^{-\overline{K}_t}} \mathbf{1}_{\{I_t > -x\}}\right] \\
&=\frac{1}{ \widehat{V}(x)}\E_{(z,0)} \left[  \widehat{V}(\overline{K}_t+x)\mathbf{1}_{\{I_t > -x\}}\E_{(z,0)}\Big[e^{-\lambda e^{-x} Z_t e^{-\overline{K}_t}} \Big| \mathcal{F}^{(e)}_t\Big]\right] \\
&=\frac{1}{ \widehat{V}(x)} \E_{(z,0)}  \left[  \widehat{V}(\overline{K}_t+x)\mathbf{1}_{\{I_t > -x\}}e^{- v_t(0,\lambda e^{-x},\overline{K})} \right]\\
&=\frac{1}{ \widehat{V}(x)} \E_{(z,x)}  \left[  \widehat{V}(\overline{K}_t)\mathbf{1}_{\{I_t > 0\}}e^{- v_t(0,\lambda e^{-K_0},\overline{K}-K_0)} \right]\\
 &=  \mathbb{E}^{(e), \uparrow}_x \left[ e^{- zv_t(0,\lambda e^{-K_0},\overline{K}-K_0) } \right].
 \end{split}
 \]
By letting $\lambda$ go to infinity, we get
 $$ \P^\uparrow_{(z,x)} (Z_t=0) =  \mathbb{E}^{(e), \uparrow}_x\left[ e^{-z v_t(0,\infty,\overline{K}-K_0) } \right]. $$
From the previous identity, it is clear that 
\[
0<\P^\uparrow_{(z,x)} (Z_t=0) \qquad \textrm{if and only if}  \qquad \mathbb{P}^{(e), \uparrow}_x\left( v_t(0,\infty,\overline{K}-K_0)<\infty\right)>0.
\]
Therefore, in order to deduce the last statement it is enough to show that Grey's condition \eqref{GreysCond} is necessary and sufficient for $\mathbb{P}^{(e), \uparrow}_x(v_t(0,\infty,\overline{K}-K_0)<\infty)>0$.

We   first observe from the Wiener-Hopf factorisation \eqref{WHfactors} applied to the spectrally positive L\'evy process associated to the branching mechanism  $\psi_0$,  that there exists a  non decreasing function $\Phi$ (which is associated to its ascending ladder height) satisfying,
\begin{equation}\label{WHbranching} \psi_0(\lambda)=\lambda \Phi(\lambda)\qquad \textrm{for} \quad \lambda\ge 0. 
\end{equation}
More precisely, from \eqref{defpsi},  \eqref{psi0} and integration by parts, we have
\begin{equation}\label{def_Phi}
\begin{split}
  \Phi(\lambda)= \gamma^2 \lambda +\int_{(0,\infty)}\frac{ e^{-\lambda x}-1 + \lambda x  }{\lambda} \mu(\ud x)\\
  =\gamma^2 \lambda +\int_{(0,\infty)} \Big(1-e^{-\lambda x}\Big) \overline{\mu}( x)\ud x,
  \end{split}
\end{equation}
where $\overline{\mu}( x):=\mu((x, \infty))$. 
Since $\Phi$ is the Laplace exponent of a subordinator, it is well-known that for any $\lambda>0$ and $k>1$, we have $\Phi(\lambda)\le k\Phi(\lambda/k)$ (see for instance the proof of Proposition III.1 in \cite{bertoin1998levy}). In particular, from \eqref{backward} and under the event that $\{t< \tau_{-x}^-\}$, we have
\begin{align*}
 \frac{\partial}{\partial s} v_t(s,\lambda e^{-x},\overline{K}-x)&=e^{\overline{K}_s-x}\psi_0(e^{-\overline{K}_s+x}v_t(s,\lambda e^{-x},\overline{K}-x)) \\
 &= v_t(s,\lambda e^{-x},\overline{K}-x)\Phi(e^{-\overline{K}_s+x}v_t(s,\lambda e^{-x},\overline{K}-x))\\
 &\geq  e^{-\overline{K}_s+x}\psi_0(v_t(s,\lambda e^{-x},\overline{K}-x)).
\end{align*}
 This
entails
$$  \int_{v_t(0,\lambda e^{-x},\overline{K}-x)}^{\lambda e^{-x}}\frac{\ud u}{\psi_0(u)} \geq \int_0^t e^{-\overline{K}_s +x} \ud s. 
$$
Assuming that  \eqref{GreysCond} holds, we deduce
\begin{equation}\label{tomate}  \int_{v_t(0,\infty,\overline{K}-x)}^{\infty}\frac{\ud u}{\psi_0(u)} \geq \int_0^t e^{-\overline{K}_s+x} \ud s,
\end{equation}
which clearly implies that $v_t(0,\infty,\overline{K}-K_0)<\infty$ with positive probability under  $\mathbb{P}^{(e), \uparrow}_x$.

Next, we assume that $\mathbb{P}^{(e), \uparrow}_x(v_t(0,\infty,\overline{K}-K_0)<\infty)>0$. Since $\psi_0$ is non-decreasing, we deduce, under the event that $\{t< \tau_{-x}^-\}$, that 
\[
\begin{split}
 \frac{\partial}{\partial s} v_t(s,\lambda e^{-x},\overline{K}-x)&=e^{\overline{K}_s-x}\psi_0(e^{-\overline{K}_s+x}v_t(s,\lambda e^{-x},\overline{K}-x))\\
 & \leq e^{\overline{K}_s-x}\psi_0(v_t(s,\lambda e^{-x},\overline{K}-x)),
\end{split}
\]
which implies 
\[
\int_{v_t(0,e^{-x},\overline{K}-x)}^{\lambda e^{-x}}\frac{\ud u}{\psi_0(u)} \leq \int_0^t e^{\overline{K}_s-x} \ud s.\]
Therefore, by letting $\lambda$ goes to $\infty$, we have 
\[
\int_{v_t(0,\infty,\overline{K}-x)}^{\infty}\frac{\ud u}{\psi_0(u)} \leq \int_0^t e^{\overline{K}_s-x} \ud s,
\]
with positive probability under  $\mathbb{P}^{(e), \uparrow}_x$.
It  implies that Grey's condition \eqref{GreysCond} holds and completes the proof.
\end{proof}

 Actually, from Grey's condition \eqref{GreysCond} and inequality \eqref{tomate}, we can deduce a nice lower bound for the probability of extinction. Indeed, let us introduce
\[
f(t):=\int_t^\infty \frac{\ud u}{\psi_0(u)}, \qquad \textrm{for}\quad t>0,
\]
and note that  the function $f:(0,\infty)\to(0,\infty)$ is a decreasing bijection and thus its inverse exists. We denote this inverse by  $\varphi$. Therefore from \eqref{tomate}, we get
\[
v_t(0,\infty, \overline{K}-x)\le \varphi\left(\int_0^t e^{-\overline{K}_s+x} \ud s\right).
\]
In other words,
\[
\mathbb{E}^{(e), \uparrow}_x \left[ \exp\left\{-z \varphi\left(\int_0^t e^{-\overline{K}}_s\right)\ud s \right\} \right]\le\mathbb{E}^{(e), \uparrow}_x\left[ e^{-z v_t(0,\infty,\overline{K}-K_0) } \right],
\]
implying 
\[
0<\mathbb{E}^{(e), \uparrow}_x \left[ \exp\left\{-z \varphi\left(\int_0^\infty e^{-\overline{K}_s}\right)\ud s \right\} \right]\le\lim_{t\to\infty}\P^\uparrow_{(z,x)} (Z_t=0),
\] 
since $\varphi$ is non-increasing and 
\begin{equation}\label{finitexp}
\int_0^\infty e^{-\overline{K}_s}\ud s <\infty, \qquad \mathbb{P}^{(e), \uparrow}_x\mathrm{-a.s.}
\end{equation}
 The claim in \eqref{finitexp} follows from the following argument. From  Theorem VI.20 in Bertoin \cite{bertoin1998levy}, we observe
 \[
 \begin{split}
\mathbb{E}^{(e), \uparrow}_x \left[ \int_0^\infty  e^{-\overline{K}_s} \ud s\right]&= 
 \int_0^\infty \mathbb{E}^{(e), \uparrow}_x \left[   e^{-\overline{K}_s} \right]\ud s\\
 &= \int_0^\infty  \frac{1}{ \widehat{V}(x)}\mathbb{E}^{(e)}_x\left[  e^{-\overline{K}_s}  \widehat{V}(\overline{K}_s)\mathbf{1}_{\{I_s >  0\}}\right]\ud s\\
 &= \frac{1}{ \widehat{V}(x)}\mathbb{E}^{(e)}_x\left[ \int_0^{\tau^-_0}  e^{-\overline{K}_s}  \widehat{V}(\overline{K}_s)\ud s\right]\\
 &= \frac{\mathbf{c}}{ \widehat{V}(x)}\int_{[0,\infty)} V(\ud y) \int_{[0,x]} \widehat{V}(\ud z) e^{-x-y+z} \widehat{V}(x+y-z),\\
 \end{split}
 \]
 where we recall that $ V$ denotes the renewal measure associated to the ascending ladder height and $\mathbf{c}$ is a constant that only depends on the 
 normalisation of the local times $L$ and $\widehat{L}$. For the sake of simplicity we take $\mathbf{c}=1$. Thus, since  $\widehat{V}$ is increasing we have 
  \[
\mathbb{E}^{(e), \uparrow}_x \left[ \int_0^\infty  e^{-\overline{K}_s} \ud s\right]\le \int_{[0,\infty)}V(\ud y) e^{-y} \widehat{V}(x+y).
 \]
The latter integral is  finite since  $\widehat{V}$ satisfies \eqref{grandO} and 
  \begin{equation}\label{Lapsharp}
 \int_{[0, \infty)}e^{-\theta x} V(\ud x)=\frac{1}{\kappa(0,\theta)}, \qquad \textrm{for} \quad \theta>0,
 \end{equation}
 which follows from the definition of $ V$ (see \eqref{ascren}) and similar arguments as in \eqref{LaplaceV}. In other words the claim in \eqref{finitexp} holds.

On the other hand, for our purposes, we are interested in conditions which guarantee that 
\[
\lim_{t\to\infty}\P^\uparrow_{(z,x)} (Z_t=0)<1.
\]
This problem is similar to determining when the probability of survival of a CSBP process  in L\'evy environments that drifts to $+\infty$, is positive. According to Proposition 2 in \cite{palau2017continuous}, the latter holds under a $x\log (x)$  moment condition on the measure $\mu$.

Assumption {(H2)} is very similar to the previous condition and implies that $Z$
has a positive probability to survive when living in a ``favorable" environment, or in other words when the 
running infimum of the L\'evy environment is positive.
\begin{proposition}\label{posi_CSBP}
If  condition {(\bf H2)} holds then  
 $$
 \lim_{t\to\infty}\P^\uparrow_{(z,x)} (Z_t>0) >0.
 $$
\end{proposition}
\begin{proof} Let us assume that condition {\bf (H2)} holds. We follow similar ideas as in the proof of Proposition 2 in \cite{palau2017continuous}.
First recall that the function $s \mapsto v_t(s,\lambda  e^{-x},\overline{K}-x)$ is non-decreasing on $[0,t]$ since $\psi_0$ is positive. Hence for any $s\in[0,t]$, we have 
$$ v_t(s,\lambda e^{-x},\overline{K} -x) \leq \lambda  e^{-x}. $$
In particular, from \eqref{backward} we have
\begin{align*}
 \frac{\partial}{\partial s} v_t(s,\lambda e^{-x},\overline{K} -x)&=e^{\overline{K}_s-x}\psi_0(e^{-\overline{K}_s+x}v_t(s,\lambda  e^{-x},\overline{K} -x)) \\
 &= v_t(s,\lambda e^{-x},\overline{K}-x)\Phi(e^{-\overline{K}_s+x}v_t(s,\lambda e^{-x},\overline{K} -x))
\\
& \leq v_t(s,\lambda e^{-x},\overline{K} -x) \Phi(e^{-\overline{K}_s}\lambda),
\end{align*}
as $\Phi$ is non-decreasing and $v_t(s,\lambda e^{-x},\overline{K} -x)$ is non-decreasing with $s$ and equals $\lambda e^{-x}$ when $s=t$. This
entails
$$  v_t(0,\lambda e^{-x},\overline{K} -x) \geq \lambda  e^{-x} \exp \left\{ - \int_0^t \Phi\left( \lambda e^{-\overline{K}_s} \right)\ud s \right\}. $$
Thus, for any $\lambda \ge 0$ 
\[
\begin{split}
\P^\uparrow_{(z,x)} (Z_t=0)&=\mathbb{E}^{(e), \uparrow}_x \left[ e^{-z v_t(0,\infty,\overline{K} -x) } \right]\\
&=\frac{1}{ \widehat{V}(x)}\mathbb{E}^{(e)}_x\left[  \widehat{V}(\overline{K}_t)e^{-z v_t(0,\infty,\overline{K} -x) }\mathbf{1}_{\{t<   \tau^-_0\}} \right]\\
&\leq \frac{1}{\widehat{V}(x)}\mathbb{E}^{(e)}_x\left[  \widehat{V}(\overline{K}_t)e^{-z v_t(0,\lambda e^{-x},\overline{K} -x) }\mathbf{1}_{\{t<  \tau^-_0\}} \right]\\
&  \leq  \frac{1}{ \widehat{V}(x)}\mathbb{E}^{(e)}_x \left[  \widehat{V}(\overline{K}_t) e^{-z \lambda  e^{-x} \exp \left\{ - \int_0^t \Phi\left( \lambda e^{-\overline{K}_s} \right)\ud s \right\} }\mathbf{1}_{\{t<  \tau^-_0\}} \right]\\
&=\mathbb{E}^{(e), \uparrow}_x  \left[ e^{-z  \lambda  e^{-x}\exp \left\{ - \int_0^t \Phi\left( \lambda e^{-\overline{K}_s} \right)\ud s \right\} } \right],
\end{split}
\]
where we have used that $\lambda \mapsto v_t(0,\lambda ,\overline{K}-x)$ is non-decreasing, see Proposition 2.2 in He et al. \cite{he2016continuous}.
Hence, we have
\[
\lim_{t\to\infty}\P^\uparrow_{(z,x)} (Z_t=0)\le \mathbb{E}^{(e), \uparrow}_x  \left[ e^{-z  \lambda  e^{-x}\exp \left\{ - \int_0^\infty \Phi\left( \lambda e^{-\overline{K}_s} \right)\ud s \right\} } \right].
\]
If
\begin{equation}\label{finitude} \mathbb{E}^{(e), \uparrow}_x \left[ \int_0^\infty \Phi\left( \lambda e^{-\overline{K}_s} \right)\ud s\right]<\infty ,\end{equation}
we get
 $$  \exp \left\{ - \int_0^{\infty} \Phi\left( \lambda e^{-\overline{K}_s} \right)\ud s \right\} >0 , \quad \mathbb{P}^{(e), \uparrow}_x -\textrm{a.s.,} $$
 and $\lim_{t\rightarrow \infty} \P^\uparrow_{(z,x)} (Z_t=0)<1. $
In other words, in order to deduce our result it is enough to show that  \eqref{finitude} holds.  We proceed similarly as in  the proof of \eqref{finitexp}.  From  the definition of $\mathbb{P}^{(e), \uparrow}_x$ and Theorem VI.20 in Bertoin \cite{bertoin1998levy}, we observe
 \[
 \begin{split}
\mathbb{E}^{(e), \uparrow}_x \left[ \int_0^\infty \Phi\left( \lambda e^{-\overline{K}_s} \right)\ud s\right]&=  \frac{1}{ \widehat{V}(x)}\mathbb{E}^{(e)}_x\left[ \int_0^{\tau^-_0} \Phi\left( \lambda e^{-\overline{K}_s} \right) \widehat{V}(\overline{K}_s)\ud s\right]\\
 &= \frac{1}{ \widehat{V}(x)}\int_{[0,\infty)} V(\ud y) \int_{[0,x]}  \widehat{V}(\ud z)\Phi\left(\lambda e^{-x-y+z}\right) \widehat{V}(x+y-z).\\
 \end{split}
 \] 
Recalling the definition of  $\Phi$ in \eqref{def_Phi} and observing that it is increasing, as well as the renewal 
function $\widehat{V}$, we obtain that
 \[
 \begin{split}
 \frac{1}{\widehat{V}(x)}\int_{[0,\infty)} V(\ud y) \int_{[0,x]} & \widehat{V}(\ud z)\Phi\left(\lambda e^{-x-y+z}\right) \widehat{V}(x+y-z)\\
 &\le \int_{[0,\infty)}V(\ud y) \Phi\left(\lambda e^{-y}\right) \widehat{V}(x+y) \\
 &= \gamma^2\lambda \int_{[0,\infty)} V(\ud y)  e^{-y} \widehat{V}(x+y)\\
 &\hspace{.5cm}+\int_{[0,\infty)} V(\ud y)  \widehat{V}(x+y)\int_{(0,\infty)} \Big(1-e^{-\lambda e^{-y}z}\Big) \overline{ \mu} (z) \ud z.
 \end{split}
 \]
 The first integral of the right-hand side is finite from identity \eqref{Lapsharp} and 
 since $\widehat{V}$ satisfies \eqref{grandO}.  For the second  integral, we first rewrite
 \[
 \begin{split}
 \int_{[0,\infty)}V(\ud y) & \widehat{V}(x+y)\int_{(0,\infty)} \Big(1-e^{-\lambda e^{-y}z}\Big)\overline{ \mu} (z) \ud z\\
 &=\int_{(0,\infty)}\ud z \overline{\mu}( z)\int_{[0,\infty)} V(\ud y)  \left(1-e^{-\lambda e^{-y}z} \right) \widehat{V}(x+y)\\
 &=\int_{(0,\infty)}\ud z \overline{\mu}( z)g(z),
 \end{split}
 \]
with 
\begin{equation}\label{defg}
g(z):=\int_{[0,\infty)} V(\ud y) \left(1-e^{-\lambda e^{-y}z} \right) \widehat{V}(x+y).
 \end{equation}
In order to conclude our proof, we need to show that under  condition {\bf (H2)}, the integral of $z\mapsto  \overline{\mu}(z)g(z)$  is finite.
In other words, we need to study the behaviour of $g(z)$ when $z$ is close to $0$ and to $\infty$. With this aim in mind, we  use that $\widehat{V}$ is subadditive and  identity \eqref{Lapsharp}, as well as the following inequality, 
 $$ 1 -e^{-z} \leq 1 \wedge z, $$
 which holds for every $z>0$.
 For  $z$  small enough, and using inequality \eqref{grandO}, we get
 \[
 \begin{split}
  g(z)&\leq \lambda z\int_{[0,1)} V(\ud y)  \widehat{V}(x+y)  e^{-y}+\lambda z\int_{[1,\infty)} V(\ud y) \widehat{V}(x+y)  e^{-y}\\
  &\leq \lambda z \left(   \widehat{V}(x+1)  V(1)+ C(x)\int_{[1,\infty)} V(\ud y) y e^{-y} \right)\\
  &\leq C_1(x)\lambda z, \end{split}
 \]
where $C(x)$ and $C_1(x)$ are two  finite constants that only depend on $x$.

For $z$ large  enough, we split the integral in \eqref{defg} into three terms. To be more precise, 
\[
 \begin{split}  g(z) &\leq \int_{[0,\infty)} V(\ud y)  \widehat{V}(x+y) (\lambda z e^{-y} \wedge 1)\\ 
 & \leq \int_{[0,1)}  V(\ud y)\widehat{V}(x+y)+\int_{[1,2\ln(\lambda z))} V(\ud y)  \widehat{V}(x+y)\\ 
 &\hspace{5cm}+\lambda z\int_{[2\ln(\lambda z), \infty)}  V(\ud y) \widehat{V}(x+y) e^{-y}.
 \end{split}
 \]
 We study  the three terms from above separately.  First, it is clear that the first term satisfies
 $$ \int_{[0,1)} V(\ud y)  \widehat{V}(x+y)  \leq V(1)  \widehat{V}(x+1).  $$
For the third term, we use \eqref{grandO} 
and deduce
\[
 \begin{split}
  \lambda z  \int_{2\ln(\lambda z)}^\infty  V(\ud y)  \widehat{V}(x+y) e^{-y}
  &\leq C_1 \lambda z e^{-2\ln(\lambda z)/2}\int_{2\ln(\lambda z)}^\infty  V(\ud y) (x+y) e^{-y/2} \\ 
  &\leq C_1  (x+1)\int_{0}^\infty  V(\ud y) (1+y) e^{-y/2}
  \end{split}
  \]
  where 
  $$\int_{0}^\infty  V(\ud y) (1+y) e^{-y/2}\leq \sum_{i\geq 0}   V ([i,i+1)) (1+i+1)e^{-i/2}\leq C_4\sum_{i\geq 0} (i+2)^2e^{-i/2}<\infty,$$
with $C_4>0$ such that 
  \[
 V(x)\le C_4 x, \qquad\textrm{for}\quad x\ge 0.
  \]
Finally, 
 \[
 \int_{[1,2\ln(\lambda z))}  V(\ud y)  \widehat{V}(x+y) \leq C(x) \int_{[1,2\ln(\lambda z))}yV(\ud y) \leq C_2(x) \ln^2(\lambda z),
 \]
 where $C(x)$ and $C_2(x)$ are two finite constants that only depend on $x$.
 Since condition ${\bf (H2)}$ holds,  the proof of our result is now complete.
 \end{proof}

 \section{Proof of Theorem \ref{maintheo}}

We have now collected all the necessary results to study the asymptotic behaviour of the extinction probability of $Z$. The proof of Theorem \ref{maintheo} follows from studying the event of survival at time $t$,  $\{Z_t>0\}$ in three different situations that depend on the behaviour 
of the infimum of the environment. To be more precise, we split the survival event as follows: for $z,x>0$, 
\begin{multline}\label{decomp} \P_{(z,x)}(Z_t>0) = \P_{(z,x)}(Z_t>0,I_t>0) \\
+ \P_{(z,x)}(Z_t>0,-y<I_t\leq 0)+ \P_{(z,x)}(Z_t>0,I_t\leq -y), \end{multline}
where $y>0$ will be chosen later on.  In other words, to deduce our result, we study such events separately for  $t$ sufficiently large. 
 
Our first result in this section concerns the first term in the right hand side of \eqref{decomp}. It says that this is   the leading term in \eqref{decomp}.

\begin{lemma}\label{L1preuve} Assume that assumptions {\bf (H1)} and {\bf (H2)} hold. 
For $z,x>0$,  there exists a positive constant $c(z,x)$ such that
$$ \P_{(z,x)}(Z_t>0,I_t>0) \sim c(z,x) \mathbb{P}^{(e)}_x(I_t>0) \sim c(z,x)  \widehat{V}(x)t^{-(1-\rho)}\ell(t), \quad \text{as} \quad t \to \infty, $$
where  $\ell$ is a slowly varying function at $\infty$, introduced in \eqref{infimumspitzer}.
\end{lemma}

\begin{proof} 
Since $\mathbf{1}_{\{Z_t>0\}}$ converges to $\mathbf{1}_{\{\forall s\ge 0,\,  Z_s>0\}}$,  $\P^\uparrow_{(z,x)}$-almost surely, as $t$ goes to $\infty$, we can apply Lemma \ref{lemme25AGKV} and
\[
\begin{split}
 \P_{(z,x)}(Z_t>0,I_t>0)&= \mathbb{P}_{(z,x)}(Z_t>0|I_t>0)\mathbb{P}^{(e)}_x(I_t>0) \\
 &\sim \P^\uparrow_{(z,x)}(\forall s\ge 0,\, Z_s>0)\mathbb{P}^{(e)}_x(I_t>0),\qquad \textrm{as}\quad t \to \infty.
\end{split}
\]
From Proposition \ref{posi_CSBP}, we know that 
$$ \P^\uparrow_{(z,x)}(\forall s\ge 0,\, Z_s>0)>0. $$
We conclude the proof by recalling the asymptotic behaviour in  \eqref{infimumspitzer} to deduce the second equivalence.
\end{proof}

We  now prove that the last term in the right hand side of \eqref{decomp} is negligible for $y$ large enough, under assumptions {\bf (H1)} and {\bf (H3)}.

\begin{lemma}\label{L2preuve}
Let $\eps, z,x>0$, $\delta\in(0,1)$ and assume that assumptions {\bf (H1)} and {\bf (H3)} hold.
 Then for $t$ and $y$ large enough, we have
$$ \P_{(z,x)}(Z_t>0,I_{t-\delta}<-y) \leq \eps \P_{(z,x+y)}(Z_t>0,I_{ t-\delta}>0). $$
\end{lemma}

\begin{proof}  Recall that  for $s\in[0,t]$, the functional $v_t(s,\lambda, \overline{K})$  is  the $\mathbb{P}^{(e)}$-a.s. unique solution of  the backward differential equation 
\eqref{backward}.  We also recall that 
 the quenched survival probability satisfies 
 \begin{eqnarray}\label{majquenched}
\P_{(z,x)}(Z_t>0|\overline{K})= 1- e^{-zv_t(0,\infty,\overline{K}-x)}.
 \end{eqnarray}
From assumption {\bf (H3)} and definition \eqref{backward}, we obtain that for $s \leq t$ and $\lambda \geq 0$,
 \[
 \begin{split}
\frac{ \partial}{\partial s} v_t(s,\lambda  e^{-x},\overline{K} -x)& \geq C e^{\overline{K}_s-x} \left( v_t(s,\lambda e^{-x},\overline{K}-x) e^{-\overline{K}_s+x} \right)^{\beta+1}
 \\
 &= Cv^{\beta+1}_t(s,\lambda  e^{-x},\overline{K}-x) e^{-\beta (\overline{K}_s-x)}.
 \end{split}
 \]
 This yields
 $$ \frac{1}{v^\beta_t(0,\lambda e^{-x},\overline{K}-x)}-\frac{1}{\lambda^\beta}\geq \beta C \mathcal{I}_t(\beta(\overline{K}-x)) ,$$
 where
 \[
\mathcal{I}_t(\beta (\overline{K}-x)):= \int_0^te^{-\beta (\overline{K}_s-x)}\ud s.
 \]
Letting $\lambda$ go to $\infty$, we obtain
 \begin{equation}\label{majv}
  v_t(0,\infty,\overline{K}-x)\leq \left(\beta C \mathcal{I}_t(\beta (\overline{K}_s-x))\right)^{-1/\beta} . 
 \end{equation}
 Using \eqref{majquenched} and \eqref{majv}, we get the following upper bound
\begin{multline}\label{term_to_bound}
 \P_{(z,x)}(Z_t>0,I_{t-\delta}<-y)\leq  \mathbb{E}^{(e)}_x \left[ \left(1 -e^{-z \left(\beta C \mathcal{I}_t(\beta (\overline{K}-x)) \right)^{-1/\beta}} \right)\mathbf{1}_{\{I_{t-\delta}<-y\}} \right]\\
 =\mathbb{E}^{(e)} \left[ \left(1 -e^{-z \left(\beta C \mathcal{I}_t(\beta (\overline{K})) \right)^{-1/\beta}} \right)\mathbf{1}_{\{I_{t-\delta}<-y-x\}} \right].
\end{multline}

On the other hand,  under assumption {\bf (H1)},  Theorems 2.18 and 2.20 in \cite{patie2016bernstein} guarantee that 
for $q\in(0,1)$,    
\[
\mathbb{E}^{(e)} \left[ \mathcal{I}_t(\beta \overline{K})^{-q} \right]<\infty, \qquad t>0,
\] 
and   for $F\in C_b(\mathbb{R}_+)$ 
\[
 \lim_{t\to\infty} \frac{\mathbb{E}^{(e)} \left[ \mathcal{I}_t(\beta \overline{K})^{-q} F(\mathcal{I}_t(\beta \overline{K})\right]}{\widehat{\kappa}(1/t,0)}= \int_0^\infty F(x)\nu_{q, \rho}(\ud x),
 \]
 where $\nu_{q, \rho}$ is a finite measure on $(0,\infty)$, see equation (2.46) in \cite{patie2016bernstein} for further details about $\nu_{q, \rho}$. Thus, by taking $F_z(x)=x^q(1-e^{-z C_\beta x^{-1/\beta}})$ with $C_\beta=(\beta C)^{-1/\beta}$, we deduce
 \begin{equation} \label{equiv_mathcal_I}
 \lim_{t\to\infty} \frac{\mathbb{E}^{(e)} \left[ 1 -e^{-z C_\beta\left(\mathcal{I}_t(\beta (\overline{K})) \right)^{-1/\beta}} \right]}{\widehat{\kappa}(1/t,0)}= \int_0^\infty x^q(1-e^{-zC_\beta x^{-1/\beta}})\nu_{q, \rho}(\ud x) =: C_{\beta, q, \rho}(z),
 \end{equation}
where the last notation has been introduced for the sake of readability.
Hence, in particular from \eqref{asymp_kappa_hat} we have
 \begin{equation*}
  \mathbb{E}^{(e)} \left[ 1 -e^{-z C_\beta\left(\mathcal{I}_t(\beta (\overline{K})) \right)^{-1/\beta}} \right] \sim \frac{ \Gamma(1+\rho)C_{\beta, q, \rho}(z)}{\rho} t^{\rho-1}\ell(t), 
  \qquad \textrm{as} \quad t \to \infty,
 \end{equation*}
where $\ell$ is the slowly varying function at $\infty$ introduced in \eqref{asymp_kappa_hat}. The latter  implies that there exists $t_0$ such that if $t \geq t_0$,
 \begin{equation} \label{asymp_mathcalI}  \mathbb{E}^{(e)} \left[ 1 -e^{-z C_\beta\left(\mathcal{I}_t(\beta (\overline{K})) \right)^{-1/\beta}} \right] \leq 2 \frac{\Gamma(1+\rho)C_{\beta, q, \rho}(z)}{\rho} t^{\rho-1}\ell(t).
 \end{equation}
Next,  we recall  from \eqref{infimumspitzer} that
 \begin{equation*}
 \mathbb{P}^{(e)}(I_t >-y) = \mathbb{P}^{(e)}_{y}(I_t >0)  \sim \widehat{V}(y)t^{\rho-1}\ell(t), \qquad \textrm{as} \quad t \to \infty.
 \end{equation*}
On the other hand from  Potter's Theorem (see Theorem 1.5.6 in Bingham et al. \cite{bingham1989regular}), we  deduce that  for any $A>1$ and $\delta_1>0$ there exists $t_1:=t_1(A_1, \delta_1)$ such that for $s\ge  h\ge t_1$, 
\[
 \frac{\mathbb{P}^{(e)}(I_h >-y)}{\mathbb{P}^{(e)}(I_{s} >-y)}\leq A\left( \frac{s}{h} \right)^{1-\rho+\delta_1}.
\]
Let us fix $A>1$ and $\delta_1>0$ and  introduce $\tau_{-y}=\inf\{t: \overline{K}_t\le -y\},$ the first hitting time of $-y$ by $\overline{K}$. 
The previous inequality implies that for $s\ge h\geq t_2:=t_0\lor t_1$,
\begin{align} \label{asymp_tauy}
 \mathbb{P}^{(e)}(h < \tau_{-y} \leq s) & = 
 \mathbb{P}^{(e)}(I_h >-y)-\mathbb{P}^{(e)}(I_{s} >-y) \nonumber \\& = 
 \mathbb{P}^{(e)}(I_{s} >-y)\left(\frac{\mathbb{P}^{(e)}(I_h >-y)}{\mathbb{P}^{(e)}(I_{s} >-y)}-1\right) \nonumber \\& \leq \left(A\left( \frac{s}{h} \right)^{1-\rho+\delta_1}-1\right)\mathbb{P}^{(e)}(I_{s} >-y).
\end{align}
For simplicity, we introduce the notation $\tilde{y}=y+x$.
Hence  from the property of independent increments of $\overline{K}$, we get the following sequence of inequalities, for $t\ge 3t_2$,
\[
\begin{split}
 \mathbb{E}^{(e)} \bigg[ &\left(1 -e^{-z C_\beta \left( \mathcal{I}_t(\beta (\overline{K})) \right)^{-1/\beta}} \right),\tau_{-\tilde{y}} \leq t -\delta\bigg] \\
 & \hspace{.5cm}  \leq   \mathbb{E}^{(e)} \left[ \left(1 -e^{-z C_\beta\left( \int_{\tau_{-\tilde{y}}}^{t} 
 e^{- \beta \overline{K}_s}\ud s \right)^{-1/\beta}} \right)\mathbf{1}_{\{\tau_{-\tilde{y}} \leq t -\delta\}}\right]\\
 & \hspace{.5cm} \leq  \mathbb{E}^{(e)} \left[ \left(1 -\exp\left\{-z C_\beta e^{-\tilde{y}}\left( \int_{0}^{t-\tau_{-\tilde{y}}} 
 e^{- \beta \big(\overline{K}_{\tau_{-\tilde{y}}+u}-\overline{K}_{\tau_{-\tilde{y}}}\big)}\ud s\right)^{-1/\beta}\right\} \right)\mathbf{1}_{\{\tau_{-\tilde{y}} \leq t -\delta\}}\right]\\
  &\hspace{.5cm}\leq  \mathbb{E}^{(e)} \left[ \left(1 -e^{-z C_\beta e^{-\tilde{y}} \left(\mathcal{I}_{\frac{t+t_2}{2}}(\beta \overline{K})\right)^{-1/\beta}} \right)\right] \\
 &\hspace{2cm}+\mathbb{E}^{(e)} \left[ \left(1 -e^{-z C_\beta e^{-\tilde{y}} \left(\mathcal{I}_{\delta}(\beta \overline{K})\right)^{-1/\beta}} \right)\right]  \mathbb{P}^{(e)} \left(  \frac{t-t_2}{2}<\tau_{-\tilde{y}} \leq t-\delta\right).
 \end{split}
 \]
 Thus from \eqref{asymp_mathcalI}, \eqref{asymp_tauy} and \eqref{infimumspitzer}, we have
 \[
 \begin{split}
\mathbb{E}^{(e)} \bigg[ &\left(1 -e^{-z C_\beta \left( \mathcal{I}_t(\beta (\overline{K})) \right)^{-1/\beta}} \right),\tau_{-\tilde{y}} \leq t -\delta\bigg]\\
& \le 2^{2-\rho} \frac{\Gamma(1+\rho)C_{\beta, q, \rho}(ze^{-\tilde{y}})}{\rho} (t+t_2)^{\rho-1}\ell\left(\frac{t+t_2}{2}\right)\\
 &+\left(A2^{1-\rho+\delta_1}\left( 1+\frac{t_2-\delta}{2t_2} \right)^{1-\rho+\delta_1}-1\right)\hat{V}(\tilde{y})(t-\delta)^{\rho-1}\ell(t-\delta)\\
 &\hspace{5cm}\times \mathbb{E}^{(e)} \left[ \left(1 -e^{-z C_\beta e^{-\tilde{y}} \left(\mathcal{I}_{\delta}(\beta \overline{K})\right)^{-1/\beta}} \right)\right] .
 \end{split}
 \]
 Next, we introduce
 \[
 \begin{split}
 c_1(z,\tilde{y})= &\left(2^{2-\rho} \frac{\Gamma(1+\rho)C_{\beta, q, \rho}(ze^{-\tilde{y}})}{\rho} \right .\\
 &\left. \vee \left(A2^{1-\rho+\delta_1}\left( 1+\frac{t_2-\delta}{2t_2} \right)^{1-\rho+\delta_1}-1\right)\mathbb{E}^{(e)} \left[ 1 -e^{-z C_\beta e^{-\tilde{y}} \left(\mathcal{I}_{\delta}(\beta \overline{K})\right)^{-1/\beta}} \right]\right).
 \end{split}
 \]
Therefore from \eqref{term_to_bound} and again from Potter's Theorem, we get  for $t\ge 3t_2$
 \[
 \begin{split}
  \frac{ \P_{(z,x)}(Z^\uparrow_t>0,I_{t-\delta}<-y)}{t^{\rho-1}\ell(t)}&\leq 
c_1(z,\tilde{y})\left( \frac{\ell\left(\frac{t+t_2}{2}\right)}{\ell(t)} + \left(1-\frac{\delta}{3t_2}\right)^{\rho-1} \widehat{V}(y+x)\frac{\ell\left(t-\delta\right)}{\ell(t)}  \right)\\
&\le c_1(z,\tilde{y})A \left( 2^{\delta_1} + \left(1-\frac{\delta}{3t_2}\right)^{\rho-1-\delta_1} \widehat{V}(y+x) \right).
 \end{split}
 \]
 Finally, we observe that the map $x\mapsto x^q(1-e^{-z C_\beta e^{-\tilde{y}} x^{-1/\beta}})$ is bounded and goes to $0$ as $y$ goes to $\infty$. Similarly the r.v. $1 -e^{-z C_\beta e^{-\tilde{y}} \left(\mathcal{I}_{\delta}(\beta \overline{K})\right)^{-1/\beta}}$ is bounded by one and goes to $0$, $\mathbb{P}^{(e)}$-a.s., as
$y$ goes to $\infty$. Thus by the Dominated Convergence Theorem, we have
\[
C_{\beta, q, \rho}(ze^{-\tilde{y}})=\int_0^\infty x^q(1-e^{-zC_\beta e^{-\tilde{y}}x^{-1/\beta}})\nu_{q, \rho}(\ud x)\xrightarrow[y\to \infty]{} 0,
\]
and
\[
\mathbb{E}^{(e)} \left[ 1 -e^{-z C_\beta e^{-\tilde{y}} \left(\mathcal{I}_{\delta}(\beta \overline{K})\right)^{-1/\beta}} \right]\xrightarrow[y\to\infty]{} 0.
\]
 In other words $c_1(z,\tilde{y})\to 0$, as $y$ increases. This implies that
 \[
\lim_{y\to\infty}
 \limsup_{t\to\infty} \frac{ \P_{(z,x)}(Z^\uparrow_t>0,I_{t-\delta}<-y)}{t^{\rho-1}\ell(t)} =0,
 \]
since $\widehat{V}(y)=\mathcal{O}(y)$ and 
for $y\rightarrow \infty$.
\end{proof}

Using Lemmas \ref{L1preuve} and \ref{L2preuve}, we are now able to conclude the proof of our main result.

\begin{proof}[Proof of Theorem \ref{maintheo}]
Let $z,x,\eps>0$ and  $\delta\in(0, 1)$. From Lemma \ref{L2preuve}, we can choose $y$ such that for $t$ large enough,
\[
 \P_{(z,x)}(Z_t>0,I_{t-\delta}<-y) \le \eps \P_{(z,x+y)}(Z_t>0,I_t > 0). 
 \]
Hence we deduce
\[
\begin{split} \P_{z}(Z_t>0) &=\P_{(z,x)}(Z_t>0,I_{t-\delta}>0)\\
&\hspace{1cm}+ \P_{(z,x)}(Z_t>0,-y<I_{t-\delta}\leq 0)+ \P_{(z,x)}(Z_t>0,I_{t-\delta}\leq -y)\\
&\le \P_{(z,x+y)}(Z_t>0,I_{t-\delta}>0) + \eps\P_{(z,x+y)}(Z_t>0,I_t> 0) \\
&\le \P_{(z,x+y)}(Z_{t-\delta}>0,I_{t-\delta}>0) + \eps\P_{(z,x+y)}(Z_t>0,I_t> 0) \\
&\le \left(\frac{\P_{(z,x+y)}(Z_{t-\delta}>0,I_{t-\delta}>0)}{\P_{(z,x+y)}(Z_t>0,I_t> 0)} + \eps\right)\P_{(z,x+y)}(Z_t>0,I_t> 0).
 \end{split}
 \]
 From Lemma \ref{L1preuve}, we know
\begin{equation} \label{asymp_tgrand}
\begin{split}
\P_{(z,x+y)}(Z_t>0,I_t>0) &\sim c(z,x+y) \mathbb{P}^{(e)}_{x+y}(I_t>0)\\
& \sim c(z,x+y)   \widehat{V}(x+y)t^{-(1-\rho)}\ell(t), \quad \text{as }  t \to \infty.
\end{split}
\end{equation}
From Potter's Theorem (see Theorem 1.5.6 in Bingham et al. \cite{bingham1989regular}), we  deduce that  for  any $A>1$ and $\delta_1>0$ there exists $t_1:=t_1(A_1,\delta_1)$ such that  
\[
\P_{z}(Z_t>0) \le \left(A\left(1+\frac{\delta}{t_1-\delta}\right)^{1-\rho+\delta_1} + \eps\right)\P_{(z,x+y)}(Z_t>0,I_t> 0).
 \]
In other words, for every $\eps>0$, there exists $y^\prime>0$ such that
$$ \P_{(z,y^\prime)}(Z_t>0,I_t>0) \leq \P_{z}(Z_t>0) \leq \left(A\left(1+\frac{\delta}{t_1-\delta}\right)^{1-\rho+\delta_1} + \eps\right)\P_{(z,y^\prime)}(Z_t>0,I_t> 0), $$
for some $A>1$ and $\delta_1>1$. Recall that $y'$ is a sequence which may depend on $z$ and $\eps$ and goes to  infinity as $\eps$ goes to $0$.  Thus, let us  take any sequence $y(z,\eps)$ satisfying for any $z,\eps>0$
\begin{equation} \label{prelim} 
\begin{split}
\P_{(z,y(z,\eps))}(Z_t>0,I_t>0) &\leq \P_{z}(Z_t>0)\\ 
&\leq \left(A\left(1+\frac{\delta}{t_1-\delta}\right)^{1-\rho+\delta_1} + \eps\right)\P_{(z,y(z,\eps))}(Z_t>0,I_t> 0),
\end{split} 
\end{equation}
and prove that 
$$C(z):= \lim_{\eps \to 0}c(z,y(z,\eps))   \widehat{V}(y(z,\eps))$$
exists and  is positive and finite. Dividing equation \eqref{prelim} by $t^{\rho-1}\ell(t)$ and using \eqref{asymp_tgrand}{\color{blue},} we deduce
\[
\begin{split}
0<c(z,y(z,\eps))   \widehat{V}(y(z,\eps))&\leq \liminf_{t \to \infty} \frac{ \P_{z}(Z_t>0)}{t^{\rho-1}\ell(t)}\\
&\leq \left(A\left(1+\frac{\delta}{t_1-\delta}\right)^{1-\rho+\delta_1} + \eps\right) c(z,y(z,\eps))  \widehat{V}(y(z,\eps))<\infty. 
\end{split}
\]
Now, letting $\delta$ goes to $0$ and then $\eps$ tends to $0$,  we  get
\begin{align*} 0<\limsup_{\eps\to 0}c(z,y(z,\eps))   \widehat{V}(y(z,\eps))&\leq \liminf_{t \to \infty} \frac{ \P_{z}(Z_t>0)}{t^{-(1-\rho)}\ell(t)}\\
&\leq \liminf_{\eps \to 0}\Big( A+\eps\Big)c(z,y(z,\eps))  \widehat{V}(y(z,\eps))\\ 
&= A\liminf_{\eps \to 0}c(z,y(z,\eps))   \widehat{V}(y(z,\eps))<\infty. \end{align*}
 Since $A$   can be taken arbitrarily close to 1, the inferior and superior limits (when $\eps$ goes to $0$) of the sequence $c(z,y(z,\eps))   \widehat{V}(y(z,\eps))$ are thus equal,  positive and finite. We thus deduce that this sequence has a limit $C(z)$ when $\eps$ goes to $0$, which is also the limit of
$$ \frac{ \P_{z}(Z_t>0)}{t^{-(1-\rho)}\ell(t)} \qquad   \textrm{when $t$ goes to $\infty$,}$$
 and obtain
$$ \P_{z}(Z_t>0) \sim C(z)t^{-(1-\rho)}\ell(t).$$
This completes the proof.
\end{proof}

\appendix

\section{Appendix} \label{appendix}

We provide in Appendix the proof of some technical results for the sake of completeness.

\begin{lemma}\label{conservative}
If \eqref{finitemom} holds then the process $Z$ is conservative, i.e.
\[
\mathbb{P}_z(Z_t<\infty)=1 , \qquad \textrm{for any} \qquad t\ge 0,
\]
and any starting point $z\ge 0$.
\end{lemma}
\begin{proof} Recall that there exists a functional $v_t(s,\lambda, \overline{K})$ which is  the unique solution of  the backward differential equation \eqref{backward} which  determines the law of the reweighted  process $(Z_t e^{-\overline{K}_t}, t\ge 0)$ as follows,
\begin{equation}\label{defvtslambdaK}
\mathbb{E}_{(z,x)}\Big[\exp\Big\{-\lambda Z_t e^{-\overline{K}_t}\Big\}\Big]=\mathbb{E}^{(e)}\Big[\exp\Big\{-zv_t(0,\lambda  e^{-x},\overline{K})\Big\}\Big].
\end{equation}
If we let $\lambda$ go to $0$ in the previous identity, we deduce
\[
\mathbb{P}_z\big(Z_t<\infty\big)=\underset{\lambda\downarrow 0}{\lim}\,\mathbb{E}_{(z,x)}\Big[\exp\Big\{-\lambda Z_t e^{-\overline{K}_t}\Big\}\Big]
=\mathbb{E}^{(e)}\left[\exp\left\{-z\lim_{\lambda \downarrow 0}v_t(0, \lambda e^{-x}, \overline{K})\right\}\right],
\]
where the limits are justified by monotonicity and dominated convergence. This implies that  the process $Z$  is conservative if and only if 
\[
\lim_{\lambda \downarrow 0}v_t(0,\lambda  e^{-x},\overline{K})=0,
\]
for every positive $t$.
Let us recall that the function $\Phi(\lambda)$ equals $\lambda^{-1}\psi_0(\lambda)$ and  observe that  $\Phi(0)=\psi'_0(0+)=0$ (see \eqref{WHbranching}). 
Since $\psi_0$ is convex and non-negative, we deduce that $\Phi$ is increasing. Finally, if we solve equation (\ref{backward}) with $\psi_0(\lambda)=\lambda \Phi(\lambda)$, we get 
$$v_t(s,\lambda  e^{-x},\overline{K})=\lambda e^{-x} \Expo{-\int_s^t \Phi(e^{-\overline{K}_r}v_t(r,\lambda e^{-x},\overline{K}))\ud r}.$$
Therefore, since $\Phi$ is increasing and $\Phi(0)=0$, we have
$$0\leq\underset{\lambda\rightarrow 0}{\lim}v_t(0,\lambda e^{-x},\overline{K})=\underset{\lambda\rightarrow 0}{\lim}\lambda e^{-x}\Expo{-\int_0^t \Phi(e^{-\overline{K}_r}v_t(r,\lambda  e^{-x},\overline{K}))\ud r}\leq \underset{\lambda\rightarrow 0}{\lim}\lambda e^{-x}=0,$$
implying that $Z$ is conservative.
\end{proof}

\begin{proof}[Proof of Proposition \ref{martingquenched}] By It\^o's formula, we have
\[
Z_te^{-\overline{K}_{t}}= Z_0+\int_0^t e^{-\overline{K}_s}\sqrt{2\gamma^2Z_s} \ud B^{(b)}_s+\int_0^t\int_{(0,\infty)}
\int_0^{Z_{s-}} ze^{-\overline{K}_{s-}}\widetilde{N}^{(b)}(\ud s, \ud z, \ud u),
\]
$\P$-a.s. Then, for $\P^{(e)}$ almost every $w^{(e)}$, we consider
$$Y_t^{w^{(e)}}=Y_0^{w^{(e)}}+M_t^{w^{(e)}}+N_t^{w^{(e)}}+W_t^{w^{(e)}}\qquad \P^{(b)}\textrm{-a.s.,}$$
 for any $t\geq 0$, where  $Y_t^{w^{(e)}}=Z_t(w^{(e)},.)\exp(-\overline{K}_t(w^{(e)}))$ and
\Bea
M_t^{w^{(e)}}&=&\int_0^t e^{-\overline{K}_s(w^{(e)})}\sqrt{2\gamma^2Z_s} \ud B^{(b)}_s \\
N_t^{w^{(e)}}&=&\int_0^t\int_{(0,1]}
\int_0^{Z_{s-}} ze^{-\overline{K}_{s-}(w^{(e)})}\widetilde{N}^{(b)}(\ud s, \ud z, \ud u),\\
W_t^{w^{(e)}}&=&\int_0^t\int_{[1,\infty)}
\int_0^{Z_{s-}} ze^{-\overline{K}_{s-}(w^{(e)})}\widetilde{N}^{(b)}(\ud s, \ud z, \ud u),
\Eea
are  $(\Omega^{(b)}, {\mathcal F}^{(b)}, \P^{(b)})$ local martingales. Let us now check that 
$Y^{w^{(e)}}$ is a  $(\Omega^{(b)}, {\mathcal F}^{(b)}, \P^{(b)})$ 
martingale by proving that the first moment of its supremum  on $[0,T]$ is finite, for any $T>0$. We consider the first time $\tau_N$ when $Y^{w^{(e)}}$ goes beyond $N$.
Using $\vert x \vert \leq 1+x^2$ and that $Y^{w^{(e)}}$ is bounded before the stopping time $\tau_N$, we get
$$
\E\left[\sup_{s < t\wedge \tau_N} Y_s^{w^{(e)}}\right] \leq 2+ \E\left[\sup_{s< t\wedge \tau_N} \Big(M_t^{w^{(e)}}\Big)^2\right] + \E\left[\sup_{s < t\wedge \tau_N} \Big(N_t^{w^{(e)}}\Big)^2\right]  +
\E\left[\sup_{s < t\wedge \tau_N} \Big\vert W_t^{w^{(e)}}\Big\vert  \right] .$$
Using that $\sup_{[0,T]} \vert \overline{K}(w^{(e)})\vert<\infty$,  we obtain
that  $e^{-\overline{K}(w^{(e)})}$ is bounded  before time $T$ (and the bound does not depend on $N$).  Thanks to Doob inequality applied to the stopped martingales, there exists $C_7$ (which does not depend on $N$) such that for any $t\leq T$,
\Bea
 \E\left[\sup_{s< t\wedge \tau_N } \Big(M_s^{w^{(e)}}\Big)^2 \right] &\leq & C_7 \int_0^t\E\left[\sup_{s< t\wedge \tau_N} Y_s^{w^{(e)}}\right]\ud s, \\
 \E\left[\sup_{s < t\wedge \tau_N} \Big(N_s^{w^{(e)}}\Big)^2 \right] &\leq & C_7\int_{[0,1]} z^2\mu(\ud z)  \int_0^t \E\left[\sup_{s< t\wedge \tau_N} Y_s^{w^{(e)}}\right]\ud s, \\
 \E\left[\sup_{s < t\wedge \tau_N} \Big\vert W_s^{w^{(e)}} \Big\vert \right] &\leq & C_7\int_{[1,\infty]} z\mu(\ud z)  \int_0^t \E\left[\sup_{s < t\wedge \tau_N} Y_s^{w^{(e)}}\right]\ud s.
\Eea
Then  Gronwall's Lemma ensures that
 there exists $C(T)$ such that for any $t\leq T$ and $N\geq 1$,
$\E\left[\sup_{s < t\wedge \tau_N} Y_s^{w^{(e)}}\right]\leq C(T)$. Letting  $N$ go to infinity completes the proof. 
\end{proof}

\vspace{.5cm}

{\bf Acknowledgements:} {\sl The authors are very grateful to the anonymous referees for their thorough review. This work  was partially funded by the Chair "Mod\'elisation Math\'ematique et Biodiversit\'e" of VEOLIA-Ecole Polytechnique-MNHN-F.X and ANR ABIM 16-CE40-0001. JCP  acknowledge support from  the Royal Society and CONACyT (CB-250590). 
This work was concluded whilst JCP was on sabbatical leave holding a David Parkin Visiting Professorship  at the University of Bath. 
He gratefully acknowledges the kind hospitality of the Department and University. }

\bibliographystyle{abbrv}
\bibliography{biblio}

\end{document}